\documentclass[preprint,3p,11pt,authoryear]{elsarticle}

\makeatletter
\def\ps@pprintTitle{%
 \let\@oddhead\@empty
 \let\@evenhead\@empty
 \def\@oddfoot{\centerline{\thepage}}%
 \let\@evenfoot\@oddfoot}
\makeatother

\usepackage{amssymb,amsthm,mathtools}
\mathtoolsset{showonlyrefs=true} 
\usepackage{lineno}  

\usepackage{enumerate,natbib,geometry,cancel,caption,graphicx,subcaption}
\usepackage{appendix} 
\usepackage{mathrsfs} 
\usepackage{dsfont} 

\usepackage[colorlinks]{hyperref}

\newtheorem{theorem}{Theorem}[section]
\newtheorem{proposition}[theorem]{Proposition}

\newtheorem{corollary}[theorem]{Corollary}
\newtheorem{lemma}[theorem]{Lemma}

\newtheorem{remark}[theorem]{Remark}

\makeatletter \@addtoreset{equation}{section} \makeatother

\newcommand{\N}{\mathbb{N}}
\newcommand{\R}{\mathbb{R}}

\newcommand{\QQ}{\mathbb{Q}}
\newcommand{\PP}{\mathbb{P}}
\newcommand{\EE}{\mathbb{E}}

\newcommand{\bb}[1]{\boldsymbol{#1}}

\newcommand{\OO}{\mathcal O}
\newcommand{\oo}{\mathrm{o}}
\newcommand{\leqdef}{\vcentcolon=}
\newcommand{\reqdef}{=\vcentcolon}
\newcommand{\rd}{{\rm d}}
\newcommand{\ind}{\mathds{1}}
\newcommand{\e}{\varepsilon}

\allowdisplaybreaks

\begin{document}

\begin{frontmatter}

\title{Non-asymptotic approximations for Pearson's chi-square statistic and its application to confidence intervals for strictly convex functions
 of the probability weights of discrete distributions}%

\author[a1]{Eric Bax}
\author[a3]{Fr\'ed\'eric Ouimet\corref{mycorrespondingauthor}}

\address[a1]{Verizon Media, Playa Vista (California) USA}
\address[a3]{Department of Mathematics and Statistics, McGill University, Montr\'eal (Qu\'ebec) Canada H3A 0B9\vspace{-5mm}}

\cortext[mycorrespondingauthor]{Corresponding author. Email address: frederic.ouimet2@mcgill.ca}

\begin{abstract}
In this paper, we develop a non-asymptotic local normal approximation for multinomial probabilities. First, we use it to find non-asymptotic total variation bounds between the measures induced by uniformly jittered multinomials and the multivariate normals with the same means and covariances. From the total variation bounds, we also derive a comparison of the cumulative distribution functions and quantile coupling inequalities between Pearson's chi-square statistic (written as the normalized quadratic form of a multinomial vector) and its multivariate normal analogue. We apply our results to find confidence intervals for the negative entropy of discrete distributions. Our method can be applied more generally to find confidence intervals for strictly convex functions of the weights of discrete distributions.
\end{abstract}


\end{frontmatter}

\section{Introduction}\label{sec:intro}

A fundamental question in machine learning and statistics is this: Given a set of i.i.d.\ observations, what can we infer about the distribution that generated the observations? In general, we cannot infer the exact distribution -- it is an intractable inverse problem. However, in some cases, if we know that the generating distribution belongs to a parametric family of distributions, then it is often possible to find a confidence set on the parameters for a given significance level. If the confidence set has a shape that makes it amenable to optimization (being convex, for example), then we can use it to infer confidence bounds on functions of the parameters by minimizing/maximizing the functions over the confidence set. In this paper, we are interested in the particular case where the parametric family is a given (fixed) discrete distribution with known values $v_1,\dots,v_{d+1}$. The parameters are the probability weights $p_{0,i}$ associated with each value $v_i$. In other words, given a sequence of i.i.d.\ observations $X_1,X_2,\dots,X_n$, we assume that
\begin{equation}\label{eq:distribution}
\PP(X_i = v_k) = p_k, \quad 1 \leq k \leq d+1,
\end{equation}
where the categorical values $v_i$ are fixed and known, but the categorical probabilities $p_i$ are unknown parameters.
We want confidence bounds on the quantity
\begin{equation}
\lambda_0 = f_{v_1,\dots,v_{d+1}}(p_{0,1},\dots,p_{0,d}),
\end{equation}
where the objective function $(p_1,\dots,p_d)\mapsto f_{v_1,\dots,v_{d+1}}(p_1,\dots,p_d)$ can depend on the categorical values $v_1,\dots,v_{d+1}$ and is assumed to be strictly convex. If we sample observations from the distribution \eqref{eq:distribution}, then note that the probability weights $p_{0,1},\dots,p_{0,d+1}$ are also parameters of the multinomial distribution associated with the sample count vector $\bb{N} = (N_1,N_2,\dots,N_{d+1})$ where
\begin{equation}
N_k = \sum_{i=1}^n \ind_{\{X_i = x_k\}}, \quad 1 \leq k \leq d+1.
\end{equation}
Therefore, in order to build confidence bounds for $\lambda_0$, an idea is to find a confidence set for the probability parameters of the sample count vector's multinomial distribution and then minimize/maximize the objective function over the probability parameters in that set. This problem is described in detail and solved in Section~\ref{sec:application}.

Motivated by the above machine learning problem, we derive in this paper closely related theoretical results that may be of independent interest. They include a non-asymptotic local normal approximation for multinomial probabilities, a non-asymptotic total variation bound between the measures induced by uniformly jittered multinomials and the corresponding multivariate normals with the same means and covariances, and a non-asymptotic comparison of the cumulative distribution functions (c.d.f.s) between Pearson's chi-square statistic (written as the normalized quadratic form of a multinomial vector) and its multivariate normal analogue, which then leads to quantile coupling bounds as a corollary. (For a good introduction to quantile coupling bounds, see, e.g., \cite{MR3007210}.)

To the best of our knowledge, the first two results are new. The local approximation result complements the asymptotic version obtained in Lemma~2.1 of \cite{MR750392}, via a Stieltjes integral representation of probabilities due to \cite{MR14626}, and later rederived independently in Theorem~2.1 of \cite{MR4249129} in a log-ratio form via elementary Taylor expansions and Stirling's formula. The total variation bound complements the asymptotic version obtained in the proof of Theorem~1 in \cite{MR1922539} and its improvement in Lemma~3.1 of \cite{MR4249129} where the inductive part of Carter's proof was removed by using the multinomial/normal local approximation directly instead of having a multi-scale comparison between binomials/normals. Unfortunately, the comparison of the c.d.f.s (and thus the quantile coupling result), being of order $(\log n)^{3/2} n^{-1/2}$ turns out to be (slightly) sub-optimal, although the derivation is easy to follow and provides a completely different route than usual (so there should still be some interest in the argument). Earlier asymptotic results in that direction can be found in \cite{MR750392} and \cite{MR752006}. Some lapses were later found in the asymptotic expansions of those two papers and corrected in \cite{MR2499200} (see also \cite{MR3079145}). Slightly better bounds, of order $n^{-1/2}$, can be obtained via a multivariate Berry-Esseen theorem, see the Kolmogorov distance bounds of \cite{Mann1997_convergence,Mann1997_Stein} mentioned on p.734 of \cite{MR3655852} and references therein. \cite{MR3655852} themselves derive, using Stein's method, distributional bounds of order $n^{-1}$ over smooth test functions between Pearson's chi-square statistic and its asymptotic chi-square distribution, see also \cite{arXiv:2107.00535} and \cite{arXiv:2111.00949} for similar results with respect to the Friedman and power divergence statistics. Since indicator functions are not smooth, the rate $n^{-1/2}$ would be optimal in our context.

Here is the outline of the paper. The definitions and theoretical results are presented in Section~\ref{sec:main.result} and proved in Appendix~\ref{sec:proofs}. The solution to the problem described in the introduction is given in Section~\ref{sec:application} and related proofs are also differed to Appendix~\ref{sec:proofs}. Technical moment estimates are gathered in Appendix~\ref{sec:moments.multinomial}.

\section{Definitions and theoretical results}\label{sec:main.result}

For $d\in \N$, define the (unit) $d$-dimensional simplex and its interior by
\begin{equation}
\mathcal{S}_d \leqdef \big\{\bb{s}\in [0,1]^d : \|\bb{s}\|_1 \leq 1\big\}, \qquad \mathcal{S}_d^o \leqdef \big\{\bb{s}\in (0,1)^d : \|\bb{s}\|_1 < 1\big\},
\end{equation}
where $\|\bb{s}\|_1 \leqdef \sum_{i=1}^d |s_i|$ denotes the $\ell^1$ norm on $\R^d$. Given a positive integer $n$ and a set of probability weights $\bb{p}\in \mathcal{S}_d$, the $\mathrm{Multinomial}\hspace{0.2mm}(n\hspace{-0.5mm},\bb{p})$ probability mass function is defined by
\begin{equation}\label{eq:multinomial.pdf}
P_{n,\bb{p}}(\bb{k}) = \frac{n!}{\prod_{i=1}^{d+1} k_i!} \, \prod_{i=1}^{d+1} p_i^{k_i}, \quad \bb{k}\in \N_0^d \cap n \mathcal{S}_d,
\end{equation}
where $p_{d+1} \leqdef 1 - \|\bb{p}\|_1 \geq 0$ and $k_{d+1} \leqdef n - \|\bb{k}\|_1$. With this notation, the multinomial distribution has $(d+1)$ categories, with respective probabilities $p_i,~1 \leq i \leq d+1$. The covariance matrix for the first $d$ components of a multinomial vector is well-known to be $n \Sigma_{\bb{p}}$, where $\Sigma_{\bb{p}} \leqdef \text{diag}(\bb{p}) - \bb{p} \bb{p}^{\top}$, see, e.g., \cite[p.377]{MR2168237}. From Theorem~1 in \cite{MR1157720}, we also know that $\det(\Sigma_{\bb{p}}) = \prod_{i=1}^{d+1} p_i$. Hence, the density function of the multivariate centered Gaussian random vector that has the same covariance matrix $\Sigma_{\bb{p}}$ is defined as follows:
\begin{equation}\label{eq:phi.M}
\phi_{\Sigma_{\bb{p}}}(\bb{y}) \leqdef \frac{\exp\big(-\frac{1}{2} \bb{y}^{\top} \Sigma_{\bb{p}}^{-1} \, \bb{y}\big)}{\sqrt{(2\pi)^d \, \prod_{i=1}^{d+1} p_i}}, \quad \bb{y}\in \R^d.
\end{equation}

Throughout this section, let $\bb{p}\in \mathcal{S}_d^o$ be given, and let
\begin{equation}
\bb{K}\sim \mathrm{Multinomial}\hspace{0.3mm}(n,\bb{p}), \qquad \bb{Y}\sim \mathrm{Normal}_d(n \bb{p}, n \Sigma_{\bb{p}}).
\end{equation}
Ultimately, our main goal (Theorem~\ref{thm.cdf.comparison}) will be to compare the c.d.f.\ of Pearson's chi-square statistic, $\bb{\delta}_{\bb{K},\bb{p}}^{\top} \Sigma_{\bb{p}}^{-1} \bb{\delta}_{\bb{K},\bb{p}}^{\phantom{\top}}$, with its Gaussian analogue $\bb{\delta}_{\bb{Y}\hspace{-0.5mm},\bb{p}}^{\top} \Sigma_{\bb{p}}^{-1} \bb{\delta}_{\bb{Y}\hspace{-0.5mm},\bb{p}}^{\phantom{\top}}\sim \chi_d^2$, namely compare
\begin{equation}\label{eq:goal}
\PP(\bb{\delta}_{\bb{K},\bb{p}}^{\top} \Sigma_{\bb{p}}^{-1} \bb{\delta}_{\bb{K},\bb{p}}^{\phantom{\top}} \leq \ell) \quad \text{with} \quad \PP(\bb{\delta}_{\bb{Y}\hspace{-0.5mm},\bb{p}}^{\top} \Sigma_{\bb{p}}^{-1} \bb{\delta}_{\bb{Y}\hspace{-0.5mm},\bb{p}}^{\phantom{\top}} \leq \ell), \quad \text{for } \ell > 0,
\end{equation}
where $\bb{\delta}_{\bb{y},\bb{p}} \leqdef (\delta_{y_1,p_1},\delta_{y_2,p_2},\dots,\delta_{y_d,p_d})^{\top}$,
\begin{equation}\label{def:delta}
\delta_{y,p_i} \leqdef \frac{y - n p_i}{\sqrt{n}}, \quad y\in \R, ~1 \leq i \leq d + 1.
\end{equation}
To see that the statistic $\bb{\delta}_{\bb{K},\bb{p}}^{\top} \Sigma_{\bb{p}}^{-1} \bb{\delta}_{\bb{K},\bb{p}}^{\phantom{\top}}$ is just a rewriting of Pearson's chi-square statistic, use the fact that
\begin{equation}\label{eq:inverse.covariance}
\Sigma_{\bb{p}}^{-1} = \mathrm{diag}(\bb{p}^{-1}) + p_{d+1}^{-1} \bb{1}_d^{\phantom{\top}}\bb{1}_d^{\top}
\end{equation}
(see, e.g., \cite[Eqn.~21]{MR1157720}) together with $K_{d+1} = n - \sum_{i=1}^d K_i$ to write
\begin{align}\label{eq:Pearson.calculation}
\bb{\delta}_{\bb{K},\bb{p}}^{\top} \Sigma_{\bb{p}}^{-1} \bb{\delta}_{\bb{K},\bb{p}}^{\phantom{\top}}
&= \frac{1}{n} (\bb{K} - n \bb{p})^{\top} \left\{\mathrm{diag}(\bb{p}^{-1}) + p_{d+1}^{-1} \bb{1}_d^{\phantom{\top}}\bb{1}_d^{\top}\right\} (\bb{K} - n \bb{p}) \notag \\[1mm]
&= \sum_{i=1}^d \frac{(K_i - n p_i)^2}{n p_i} + \sum_{i=1}^d \sum_{j=1}^d \frac{(K_i - n p_i) (K_j - n p_j)}{n p_{d+1}} \notag \\
&= \sum_{i=1}^d \frac{(K_i - n p_i)^2}{n p_i} + \frac{(K_{d+1} - n p_{d+1})^2}{n p_{d+1}} = \sum_{i=1}^{d+1} \frac{(K_i - n p_i)^2}{n p_i}.
\end{align}

\begin{remark}\label{eq:remark}
Notice that
\begin{equation}
\delta_{k_{d+1},p_{d+1}} = - \sum_{i=1}^d \delta_{k_i,p_i},
\end{equation}
because $p_{d+1} = 1 - \|\bb{p}\|_1$ and $k_{d+1} = n - \|\bb{k}\|_1$.
\end{remark}

\begin{remark}\label{rem:gamma.integral}
Let $\ell > 0$ be given. Since we know that $\bb{\delta}_{\bb{Y}\hspace{-0.5mm},\bb{p}}^{\top} \Sigma_{\bb{p}}^{-1} \bb{\delta}_{\bb{Y}\hspace{-0.5mm},\bb{p}}^{\phantom{\top}}$ is $\chi_d^2$ distributed, then
\begin{equation}
\PP(\bb{\delta}_{\bb{Y}\hspace{-0.5mm},\bb{p}}^{\top} \Sigma_{\bb{p}}^{-1} \bb{\delta}_{\bb{Y}\hspace{-0.5mm},\bb{p}}^{\phantom{\top}} \leq \ell) = \overline{\gamma}(d/2, \ell/2),
\end{equation}
where the two functions
\begin{equation}
\overline{\gamma}(x,a) \leqdef \frac{\int_0^a t^{x - 1} e^{-t} \rd t}{\Gamma(x)}, \qquad \Gamma(x) = \int_0^{\infty} t^{x - 1} e^{-t} \rd t, \quad x, a > 0,
\end{equation}
denote the regularized lower incomplete gamma function and Euler's gamma function, respectively.
\end{remark}

\begin{remark}
Throughout the paper, the notation $u = \OO(v)$ means that $\limsup_{n\to \infty} |u / v| < C$, where $C\in (0,\infty)$ is a universal constant. Whenever $C$ might depend on a parameter, we add a subscript: for example, $u = \OO_d(v)$. Similarly, $u = \oo(v)$ means that $\lim_{n\to \infty} |u / v| = 0$, and subscripts indicate which parameters the convergence rate can depend on.
\end{remark}

The first step (Proposition~\ref{prop:p.k.expansion.non.asymptotic}) to achieve our goal \eqref{eq:goal} is to prove a non-asymptotic version of the local limit theorem found in \cite[Lemma~2.1]{MR750392} and \cite[Theorem~2.1]{MR4249129}. Specifically, we develop a non-asymptotic expansion for the log-ratio of the multinomial probability mass function \eqref{eq:multinomial.pdf} over the corresponding multivariate normal density \eqref{eq:phi.M}. Our approximation is uniform for $\bb{k}$ in the bulk of the $\mathrm{Multinomial}\hspace{0.2mm}(n\hspace{-0.5mm},\bb{p})$ distribution and uniform for $\bb{p}$ in a compact set away from the boundary of the simplex $\mathcal{S}_d$, namely for
\begin{align}
&\bb{k}\in \mathscr{B}_{n,\bb{p}} \leqdef \left\{\bb{k}\in \N_0^d \cap n \mathcal{S}_d^o : \max_{1 \leq i \leq d + 1} \left|\frac{\delta_{k_i,p_i}}{\sqrt{n} \, p_i}\right| \leq \tau \sqrt{\frac{\log n}{n}}\right\}, \label{eq:bulk} \\
&\bb{p}\in \mathscr{P}_{\tau} \leqdef \left\{\bb{s}\in \mathcal{S}_d^o : \max_{1 \leq i \leq d + 1} s_i^{-1} \leq \tau\right\}, \label{eq:X.tau}
\end{align}
where $\tau \geq d + 1$ is a fixed parameter.

\begin{proposition}[Non-asymptotic local expansion]\label{prop:p.k.expansion.non.asymptotic}
Let $\tau \geq d + 1$ be given. Then, uniformly for $\bb{k}\in \mathscr{B}_{n,\bb{p}}$, $\bb{p}\in \mathscr{P}_{\tau}$ and $n\geq \tau^4$, we have
\begin{equation}\label{eq:LLT.order.2.log}
\begin{aligned}
\log\left\{\frac{P_{n,\bb{p}}(\bb{k})}{n^{-d/2} \phi_{\Sigma_{\bb{p}}}(\bb{\delta}_{\bb{k},\bb{p}})}\right\} = n^{-1/2} \sum_{i=1}^{d+1} \left(\frac{1}{6 p_i^2} \, \delta_{k_i,p_i}^3 - \frac{1}{2 p_i} \, \delta_{k_i,p_i}\right) + R_{n,\bb{p}}(\bb{k}),
\end{aligned}
\end{equation}
where the error $R_{n,\bb{p}}(\bb{k})$ satisfies
\begin{equation}
\begin{aligned}
|R_{n,\bb{p}}(\bb{k})|
&\leq n^{-1} \sum_{i=1}^{d+1} \left(\frac{21}{p_i^3} \, |\delta_{k_i,p_i}|^4 + \frac{10}{p_i^2} \, |\delta_{k_i,p_i}|^2 + \frac{2 \tau}{3}\right).
\end{aligned}
\end{equation}
\end{proposition}

The following corollary may be of independent interest. It extends the local limit theorem from Proposition~\ref{prop:p.k.expansion.non.asymptotic} to one-to-one transformations of multinomial vectors.

\begin{corollary}
Let $\tau \geq d + 1$ be given. Let $\bb{h}(\cdot)$ be a one-to-one mapping from an open subset $\mathcal{D}$ of $\mathcal{S}_d$ onto a subset $\mathcal{R}$ of $\R^d$. Assume further that $\bb{h}$ has continuous partial derivatives on $\mathcal{D}$ and its Jacobian determinant $\det\{\frac{\rd}{\rd \bb{x}} \bb{h}(\bb{x})\}$ is non-zero for all $\bb{x}\in \mathcal{D}$. Then, uniformly for $\bb{y}\in \mathcal{R}$ such that $\bb{h}^{-1}(\bb{y}) \in \mathscr{B}_{n,\bb{p}}$, $\bb{p}\in \mathscr{P}_{\tau}$ and $n\geq \tau^4$, we have
\begin{equation}
\begin{aligned}
&\log\left[\frac{P_{n,\bb{p}}\{\bb{h}^{-1}(\bb{y})\} \left|\det\left\{\frac{\rd}{\rd \bb{x}} \bb{h}(\bb{x})|_{\bb{x} = \bb{h}^{-1}(\bb{y})}\right\}\right|}{n^{-d/2} \phi_{\Sigma_{\bb{p}}}(\bb{\delta}_{\bb{h}^{-1}(\bb{y}),\bb{p}})}\right] \\
&\quad= n^{-1/2} \sum_{i=1}^{d+1} \left[\frac{1}{6 p_i^2} \, \delta_{\{\bb{h}^{-1}(\bb{y})\}_i,p_i}^3 - \frac{1}{2 p_i} \, \delta_{\{\bb{h}^{-1}(\bb{y})\}_i,p_i}\right] + R_{n,\bb{p}}\{\bb{h}^{-1}(\bb{y})\},
\end{aligned}
\end{equation}
where the error $R_{n,\bb{p}}\{\bb{h}^{-1}(\bb{y})\}$ satisfies
\begin{equation}
\begin{aligned}
|R_{n,\bb{p}}\{\bb{h}^{-1}(\bb{y})\}|
&\leq n^{-1} \sum_{i=1}^{d+1} \left[\frac{21}{p_i^3} \, |\delta_{\{\bb{h}^{-1}(\bb{y})\}_i,p_i}|^4 + \frac{10}{p_i^2} \, |\delta_{\{\bb{h}^{-1}(\bb{y})\}_i,p_i}|^2 + \frac{2 \tau}{3}\right].
\end{aligned}
\end{equation}
\end{corollary}

Using the local expansion from Proposition~\ref{prop:p.k.expansion.non.asymptotic}, the second step (Proposition~\ref{prop:total.variation.bound}) to achieve our goal \eqref{eq:goal} consists in finding an upper bound on the total variation between the probability measure associated with the multivariate normal in \eqref{eq:phi.M} and the probability measure of a multinomial vector jittered by a uniform random variable on $(-1/2,1/2)^d$. The proposition is a non-asymptotic version of \cite[Lemma~3.1]{MR4249129}.

\begin{proposition}[Total variation]\label{prop:total.variation.bound}
Let $\tau\geq d + 1$ be given, and let $\bb{K}\sim \mathrm{Multinomial}\hspace{0.2mm}(n\hspace{-0.5mm},\bb{p})$ and $\bb{U}\sim \mathrm{Uniform}\hspace{0.2mm}(-1/2,1/2)^d$, where $\bb{K}$ and $\bb{U}$ are assumed independent. Define $\bb{X} \leqdef \bb{K} + \bb{U}$ and let $\widetilde{\PP}_{n\hspace{-0.5mm},\bb{p}}$ be the law of $\bb{X}$. Let $\QQ_{n,\bb{p}}$ be the law of the multivariate normal distribution $\mathrm{Normal}_d(n \bb{p}, n \Sigma_{\bb{p}})$, where recall $\Sigma_{\bb{p}} \leqdef \mathrm{diag}(\bb{p}) - \bb{p} \bb{p}^{\top}$. Then, for all $n\geq \tau^4$, we have
\begin{equation}
\sup_{\bb{p}\in \mathscr{P}_{\tau}} \|\widetilde{\PP}_{n\hspace{-0.5mm},\bb{p}} - \QQ_{n,\bb{p}}\| \leq \frac{8.03 \, \tau^{3/2} (d + 1)^{1/2}}{n^{1/2}},
\end{equation}
where $\| \cdot \|$ denotes the total variation norm.
\end{proposition}

The final step (Theorem~\ref{thm.cdf.comparison}) to achieve our goal \eqref{eq:goal} consists in controlling the difference between $\PP(\bb{\delta}_{\bb{K},\bb{p}}^{\top} \Sigma_{\bb{p}}^{-1} \bb{\delta}_{\bb{K},\bb{p}}^{\phantom{\top}} \leq \ell)$ and $\PP(\bb{\delta}_{\bb{Y}\hspace{-0.5mm},\bb{p}}^{\top} \Sigma_{\bb{p}}^{-1} \bb{\delta}_{\bb{Y}\hspace{-0.5mm},\bb{p}}^{\phantom{\top}} \leq \ell)$ using the total variation bound from Proposition~\ref{prop:total.variation.bound}.

\begin{theorem}[C.d.f.\ comparison]\label{thm.cdf.comparison}
Let $\tau\geq d + 1$ and $\ell > 0$ be given. Then, for all $n\geq \tau^4$, we have
\begin{equation}\label{eq:thm.cdf.comparison}
\sup_{\bb{p}\in \mathscr{P}_{\tau}} \big|\PP(\bb{\delta}_{\bb{K},\bb{p}}^{\top} \Sigma_{\bb{p}}^{-1} \bb{\delta}_{\bb{K},\bb{p}}^{\phantom{\top}} \leq \ell) - \PP(\bb{\delta}_{\bb{Y}\hspace{-0.5mm},\bb{p}}^{\top} \Sigma_{\bb{p}}^{-1} \bb{\delta}_{\bb{Y}\hspace{-0.5mm},\bb{p}}^{\phantom{\top}} \leq \ell)\big| \leq \frac{1.26 \, \tau^3 (d + 1) (\log n)^{3/2}}{n^{1/2}}.
\end{equation}
\end{theorem}

As a consequence of Remark~\ref{rem:gamma.integral}, Theorem~\ref{thm.cdf.comparison} and the mean value theorem, we deduce the following stochastic upper bound on $\bb{\delta}_{\bb{K}\hspace{-0.5mm},\bb{p}}^{\top} \Sigma_{\bb{p}}^{-1} \bb{\delta}_{\bb{K}\hspace{-0.5mm},\bb{p}}^{\phantom{\top}}$.

\begin{corollary}[Quantile coupling]\label{cor:multivariate.Tusnady.inequality}
Let $F$ be the cumulative distribution function (c.d.f.) of $\bb{\delta}_{\bb{K},\bb{p}}^{\top} \Sigma_{\bb{p}}^{-1} \bb{\delta}_{\bb{K},\bb{p}}^{\phantom{\top}}$, and let $F^{\star}(q) \leqdef \inf\{x\in \R : F(x) \geq q\}$ denote the generalized inverse distribution function (or quantile function) associated with $F$. Also, let $G$ be the c.d.f.\ of $\bb{\delta}_{\bb{Y}\hspace{-0.5mm},\bb{p}}^{\top} \Sigma_{\bb{p}}^{-1} \bb{\delta}_{\bb{Y}\hspace{-0.5mm},\bb{p}}^{\phantom{\top}}$, and define the random variable $\Xi_{n,\bb{p}} \leqdef F^{\star}\{G(\bb{\delta}_{\bb{Y}\hspace{-0.5mm},\bb{p}}^{\top} \Sigma_{\bb{p}}^{-1} \bb{\delta}_{\bb{Y}\hspace{-0.5mm},\bb{p}}^{\phantom{\top}})\}$. For $\ell > 0$, let
\begin{equation}
c(\ell) \leqdef \ell - \frac{1.26 \, \tau^3 (d + 1) (\log n)^{3/2}}{G'(\ell) \, n^{1/2}}, \quad G'(\ell) = \frac{(\ell/2)^{d/2 - 1} e^{-\ell/2}}{2 \, \Gamma(d/2)}.
\end{equation}
Then, with the event
\begin{equation}\label{eq:event}
A \leqdef \left\{\omega\in \Omega : 2 (d/2 - 1) \leq \bb{\delta}_{\bb{Y}(\omega),\bb{p}}^{\top} \Sigma_{\bb{p}}^{-1} \bb{\delta}_{\bb{Y}(\omega),\bb{p}}^{\phantom{\top}} \leq \sup_{\ell\in [2 (d/2 - 1),\infty)} c(\ell)\right\},
\end{equation}
we have, for $n\geq n(d,\tau)$ large enough,
\begin{equation}
\Xi_{n,\bb{p}}(\omega) \leq \bb{\delta}_{\bb{Y}(\omega),\bb{p}}^{\top} \Sigma_{\bb{p}}^{-1} \bb{\delta}_{\bb{Y}(\omega),\bb{p}}^{\phantom{\top}} + \frac{1.26 \, \tau^3 (d + 1) (\log n)^{3/2}}{G'\{L_{n,\bb{p},d,\tau}(\omega)\} \, n^{1/2}}, \quad \omega\in A,
\end{equation}
where $L_{n,\bb{p},d,\tau}(\omega)$ solves $\bb{\delta}_{\bb{Y}(\omega),\bb{p}}^{\top} \Sigma_{\bb{p}}^{-1} \bb{\delta}_{\bb{Y}(\omega),\bb{p}}^{\phantom{\top}} = c\{L_{n,\bb{p},d,\tau}(\omega)\}$.
\end{corollary}

\section{Application to confidence intervals}\label{sec:application}

\subsection{Description of the problem}\label{sec:description.problem}

In this section, we apply the method mentioned in Section~\ref{sec:intro} and use the normal approximation results we obtained in Section~\ref{sec:main.result} to clarify the optimization step, given the multi-dimensionality of the problem and the roughness of the confidence set for the multinomial.

Formally, consider a discrete random variable $X$ having the following probability mass function
\begin{equation}\label{eq:X.discrete.distribution}
\PP(X = v_i) = p_{0,i}, \quad 1 \leq i \leq d + 1,
\end{equation}
where $d\in \N$ is a positive integer, the values $(\bb{v},v_{d+1})\in \R^{d+1}$ are known, but the probability weights $p_{0,i}$ are unknown. We consider a (fixed and given) strictly convex objective function $\bb{p}\mapsto f_{\bb{v},v_{d+1}}(\bb{p})$, which may depend on the values of $\bb{v}$ and $v_{d+1}$, and our goal is to estimate and find a confidence interval for
\begin{equation}\label{eq:objective.function}
\lambda_0 \leqdef f_{\bb{v},v_{d+1}}(\bb{p}_0), \quad \text{where } \bb{p}_0 \leqdef (p_{0,1},\dots,p_{0,d}).
\end{equation}
Given a significance level $\alpha\in (0,1)$, a sequence of i.i.d.\ observations $X_1, X_2, \dots X_n$ following the discrete distribution in \eqref{eq:X.discrete.distribution}, and the associated sample count random variables
\begin{equation}
N_k \leqdef \sum_{i=1}^n \ind_{\{X_i = v_k\}}, \quad 1 \leq k \leq d + 1,
\end{equation}
we will show how to construct a lower bound $\lambda_{\star} = \lambda_{\star}(\bb{v}, v_{d+1}, N_1, N_2, \dots, N_d)$ and an upper bound $\lambda^{\star} = \lambda^{\star}(\bb{v}, v_{d+1}, N_1, N_2, \dots, N_d)$ that satisfy
\begin{equation}\label{eq:problem.description.general}
\PP(\lambda_{\star} \leq \lambda_0 \leq \lambda^{\star}) \geq 1 - \alpha + \e_n, \quad \text{for all } n\geq n_{\star},
\end{equation}
where $n_{\star}$ is an explicit threshold, $\bb{N} \leqdef (N_1, N_2, \dots, N_d)\sim \mathrm{Multinomial}\hspace{0.3mm}(n,\bb{p}_0)$, and $\e_n$ is a small explicit quantity that goes to $0$ as $n\to \infty$.

\vspace{2mm}
When there are only two category values ($d = 1$), this problem is easy. Indeed, using a binary search, a computer can find an exact $(1-\alpha)\cdot 100\%$-confidence interval for $p_0$, say $[\underline{p_0},\overline{p_0}]$, using the fact that $N_1\sim \mathrm{Binomial}(n,p_0)$. Then we find $\lambda_{\star}$ (resp., $\lambda^{\star}$) simply by minimizing (resp., maximizing) the objective function $p\mapsto f_{v_1,v_2}(p)$ over $p\in [\underline{p_0},\overline{p_0}]$, i.e.,
\begin{equation}\label{eq:compute.mu.star.d.1}
\lambda_{\star} = \min_{p\in [\underline{p_0},\overline{p_0}]} f_{v_1,v_2}(p), \qquad \lambda^{\star} = \max_{p\in [\underline{p_0},\overline{p_0}]} f_{v_1,v_2}(p).
\end{equation}
(For details on binomial inversion, see, e.g., \cite{MR2249822}. Binomial tails can be computed accurately using Loader's method \cite{Loader2002}, as used in \texttt{pbinomial()} in \texttt{R}, or Stirling's approximation with several terms for very large $n$, or computed exactly using arithmetic on extended numbers such as BigInteger and BigDecimal in Java.)

In higher dimensions however, the problem becomes much more difficult. Given a significance level $\alpha\in (0,1)$ and an (observed) vector of sample counts
\begin{equation}
n_i\in \N, ~ 1 \leq i \leq d+1, \qquad \text{where } n_{d+1} \leqdef n - \|\bb{n}\|_1 ~\text{and}~ \hat{\bb{p}}_n \leqdef \bb{n}/n,
\end{equation}
we have to optimize the objective function $\bb{p}\mapsto f_{\bb{v},v_{d+1}}(\bb{p})$ over vectors of probabilities in the following confidence set for $\bb{p}_0$:
\begin{equation}\label{eq:confidence.set.p.OG}
\mathcal{C}(\hat{\bb{p}}_n,\alpha) \leqdef \Bigg\{\bb{p}\in \mathcal{S}_d^o : \sum_{\substack{\bb{k}\in \N_0^d \cap n \mathcal{S}_d \\ P_{n,\bb{p}}(\bb{k}) \geq P_{n,\bb{p}}(\bb{n})}} \hspace{-4mm} P_{n,\bb{p}}(\bb{k}) \leq 1 - \alpha\Bigg\}.
\end{equation}

\begin{remark}\label{rem:unique}
Assuming that the (observed) sample counts $n_i$ are all positive (i.e., each of the $d+1$ categories has at least one observation), note that the sum in \eqref{eq:confidence.set.p.OG} is equal~to~$1$ for any $\bb{p}$ on the boundary of the simplex since $P_{n,\bb{p}}(\bb{n}) = 0$ in that case. This is why the confidence set $\mathcal{C}(\hat{\bb{p}}_n,\alpha)$ is restricted to the interior of the simplex in \eqref{eq:confidence.set.p.OG}. This remark will play an important role in showing that the optima which define the confidence bounds in Theorem~\ref{thm:solution} are unique, see \eqref{eq:last} and below.
\end{remark}

The shape of the set $\mathcal{C}(\hat{\bb{p}}_n,\alpha)$ is unclear and the boundaries need not be smooth, which makes the optimization of the objective function much more complicated than in the binomial case ($d = 1$). To overcome this obstacle, our strategy will be to find a smooth superset that contains the exact confidence set for $\bb{p}_0$ and then compute $\lambda_{\star}$ and $\lambda^{\star}$ by optimizing over the smooth superset, as in \eqref{eq:compute.mu.star.d.1}. The small quantity $\e_n$ in \eqref{eq:problem.description.general} is the ``price'' we have to pay in confidence to transfer the optimization from the rough (but exact) confidence set to the smooth superset. The smooth superset for $\bb{p}_0 = (p_{0,1},p_{0,2},\dots,p_{0,d})$ will be constructed using the fact that the law of the sample count vector $\bb{N}$, i.e., the $\mathrm{Multinomial}\hspace{0.3mm}(n,\bb{p}_0)$ distribution, is close in total variation, when jittered by a $\mathrm{Uniform}\hspace{0.2mm}(-1/2,1/2)^d$, to the multivariate normal distribution with the same mean $n \bb{p}$ and covariance $n \Sigma_{\bb{p}}$. This point was proved in Proposition~\ref{prop:total.variation.bound}. Other elements of the proof of Theorem~\ref{thm.cdf.comparison} will be used for the comparison of the confidence sets.

\subsection{Solution to the problem}\label{sec:solution.problem}

As mentioned previously, it is not obvious (a priori) how we can optimize the objective function $\bb{p}\mapsto f_{\bb{v},v_{d+1}}(\bb{p})$ over the confidence set $\mathcal{C}(\hat{\bb{p}}_n,\alpha)$ and find an explicit solution given the multi-dimensional, discrete and rough nature of the confidence set. However, a natural idea to simplify the problem is to replace the confidence set $\mathcal{C}(\hat{\bb{p}}_n,\alpha)$ by its Gaussian analogue, namely
\begin{equation}\label{eq:confidence.set.p.new}
\begin{aligned}
\widetilde{\mathcal{C}}(\hat{\bb{p}}_n,\alpha)
&\leqdef \left\{\bb{p}\in \mathcal{S}_d^o : \int_{\phi_{\Sigma_{\bb{p}}}(\bb{\delta}_{\bb{y},\bb{p}}) \geq \phi_{\Sigma_{\bb{p}}}(\bb{\delta}_{\bb{n},\bb{p}})} \frac{\phi_{\Sigma_{\bb{p}}}(\bb{\delta}_{\bb{y},\bb{p}})}{n^{d/2}} \rd \bb{y} \leq 1 - \alpha\right\} \\[1mm]
&= \left\{\bb{p}\in \mathcal{S}_d^o : \PP\left(\bb{\delta}_{\bb{Y}\hspace{-0.5mm},\bb{p}}^{\top} \Sigma_{\bb{p}}^{-1} \bb{\delta}_{\bb{Y}\hspace{-0.5mm},\bb{p}}^{\phantom{\top}} \leq \bb{\delta}_{\bb{n},\bb{p}}^{\top} \Sigma_{\bb{p}}^{-1} \bb{\delta}_{\bb{n},\bb{p}}^{\phantom{\top}}\right) \leq 1 - \alpha\right\} \\[1.5mm]
&= \left\{\bb{p}\in \mathcal{S}_d^o : \overline{\gamma}(d/2, \bb{\delta}_{\bb{n},\bb{p}}^{\top} \Sigma_{\bb{p}}^{-1} \bb{\delta}_{\bb{n},\bb{p}}^{\phantom{\top}}/2) \leq 1 - \alpha\right\},
\end{aligned}
\end{equation}
where recall $\bb{Y}\sim \mathrm{Normal}_d(\bb{0}_d,\Sigma_{\bb{p}})$, $\phi_{\Sigma_{\bb{p}}}$ was defined in \eqref{eq:phi.M}, $\bb{\delta}_{\bb{y},\bb{p}}$ was defined in \eqref{def:delta}, and $\overline{\gamma}(\cdot,\cdot)$ was defined in \eqref{rem:gamma.integral}.

First, we need to find a Gaussian superset $\widetilde{\mathcal{C}}(\hat{\bb{p}}_n, \alpha - \e_n)$ as in \eqref{eq:confidence.set.p.new} that contains the exact confidence set $\mathcal{C}(\hat{\bb{p}}_n,\alpha)$, where $\e_n$ is a positive quantity that goes to zero as $n\to \infty$. The proposition below does just that. The proof relies on the total variation bound from Proposition~\ref{prop:total.variation.bound} and on other elements of the proof of Theorem~\ref{thm.cdf.comparison}.

\begin{proposition}\label{prop:optimization.problem.before}
Let $\alpha\in (0,1)$, $\tau\geq d + 1$, and let $\bb{n}\in n \mathcal{S}_d^o$ be given such that $\hat{\bb{p}}_n\in \mathscr{P}_{\tau}$. Recall the definition of the confidence sets $\mathcal{C}(\hat{\bb{p}}_n,\alpha)$ and $\widetilde{\mathcal{C}}(\hat{\bb{p}}_n,\alpha)$ in \eqref{eq:confidence.set.p.OG} and \eqref{eq:confidence.set.p.new}. For all $n\geq \tau^4$, we have
\begin{equation}
\mathcal{C}(\hat{\bb{p}}_n,\alpha) \subseteq \widetilde{\mathcal{C}}(\hat{\bb{p}}_n, \alpha - \e_n), \quad \text{with} \quad \e_n \leqdef \frac{194.96 \, \tau^3 (d + 1) (\log n)^{3/2}}{n^{1/2}}.
\end{equation}
\end{proposition}

Finally, by Proposition~\ref{prop:optimization.problem.before}, we can find the lower bound $\lambda_{\star}$ and the upper bound $\lambda^{\star}$ by optimizing the objective function $f_{\bb{v},v_{d+1}}(\bb{p})$ over the Gaussian superset $\widetilde{\mathcal{C}}(\hat{\bb{p}}_n, \alpha - \e_n)$. This solves the problem mentioned in the introduction and described in Section~\ref{sec:description.problem}. The proof shows that the Gaussian superset is strictly convex and it also establishes the existence and uniqueness of the minimizers/maximizers.

\begin{theorem}[Existence and uniqueness of the solution to the problem]\label{thm:solution}
Assume the objective function $\bb{p}\mapsto f_{\bb{v},v_{d+1}}(\bb{p})$ to be strictly convex. Let $\alpha\in (0,1)$, $\tau\geq d + 1$, and let $\bb{n}\in n \mathcal{S}_d^o$ be given such that $\hat{\bb{p}}_n\in \mathscr{P}_{\tau}$. We have
\begin{equation}\label{eq:problem.description.general.next}
\PP(\lambda_{\star} \leq \lambda_0 \leq \lambda^{\star}) \geq 1 - \alpha + \e_n, \quad \text{for all } n\geq \tau^4,
\end{equation}
with the choices
\begin{equation}\label{eq:compute.mu.star.d.general}
\lambda_{\star} = \min_{\bb{p}\in \widetilde{\mathcal{C}}(\hat{\bb{p}}_n, \alpha - \e_n)} f_{\bb{v},v_{d+1}}(\bb{p}), \qquad \lambda^{\star} = \max_{\bb{p}\in \widetilde{\mathcal{C}}(\hat{\bb{p}}_n, \alpha - \e_n)} f_{\bb{v},v_{d+1}}(\bb{p}).
\end{equation}
Furthermore, the set $\widetilde{\mathcal{C}}(\hat{\bb{p}}_n, \alpha - \e_n)$ is strictly convex, so that the minimizers/maximizers in \eqref{eq:compute.mu.star.d.general} are well defined (i.e., unique).
\end{theorem}

The bound on the confidence level in \eqref{eq:problem.description.general.next} cannot be close to exact unless the number of observations $n$ is extremely high. Therefore, for practical purposes, we simply let $\e_n = 0$. This is illustrated below for two examples:
\begin{align}
&f_{\bb{v},v_{d+1}}^{[1]}(\bb{p}) = \sum_{i=1}^{d+1} p_i \log p_i, \qquad (\text{negative entropy function}) \\
&f_{\bb{v},v_{d+1}}^{[2]}(\bb{p}) = \bb{p}^{\top} A \bb{p}, \qquad (\text{quadratic form})
\end{align}
where $A$ is a $d\times d$ positive definite matrix that may depend on the known categorical values $\bb{v}$ and $v_{d+1}$.

In Figures~\ref{fig:example.1}~and~\ref{fig:example.2}, the confidence bounds are shown to converge to the value of the objective functions at $p_0$ (i.e., $\lambda_0$) and the empirical significance levels are shown to stay below the nominal level $\alpha = 0.05$. The \texttt{R} code that generated the figures is available at \href{https://www.dropbox.com/s/tfwnnba642740xo/multinomial_method.R?dl=0}{https://www.dropbox.com/s/tfwnnba642740xo/multinomial\_method.R?dl=0}.

\vspace{4mm}
\begin{figure}[!ht]
\captionsetup{width=0.8\linewidth}
\centering
\begin{subfigure}[b]{0.450\textwidth}
\centering
\includegraphics[width=\textwidth, height=0.85\textwidth]{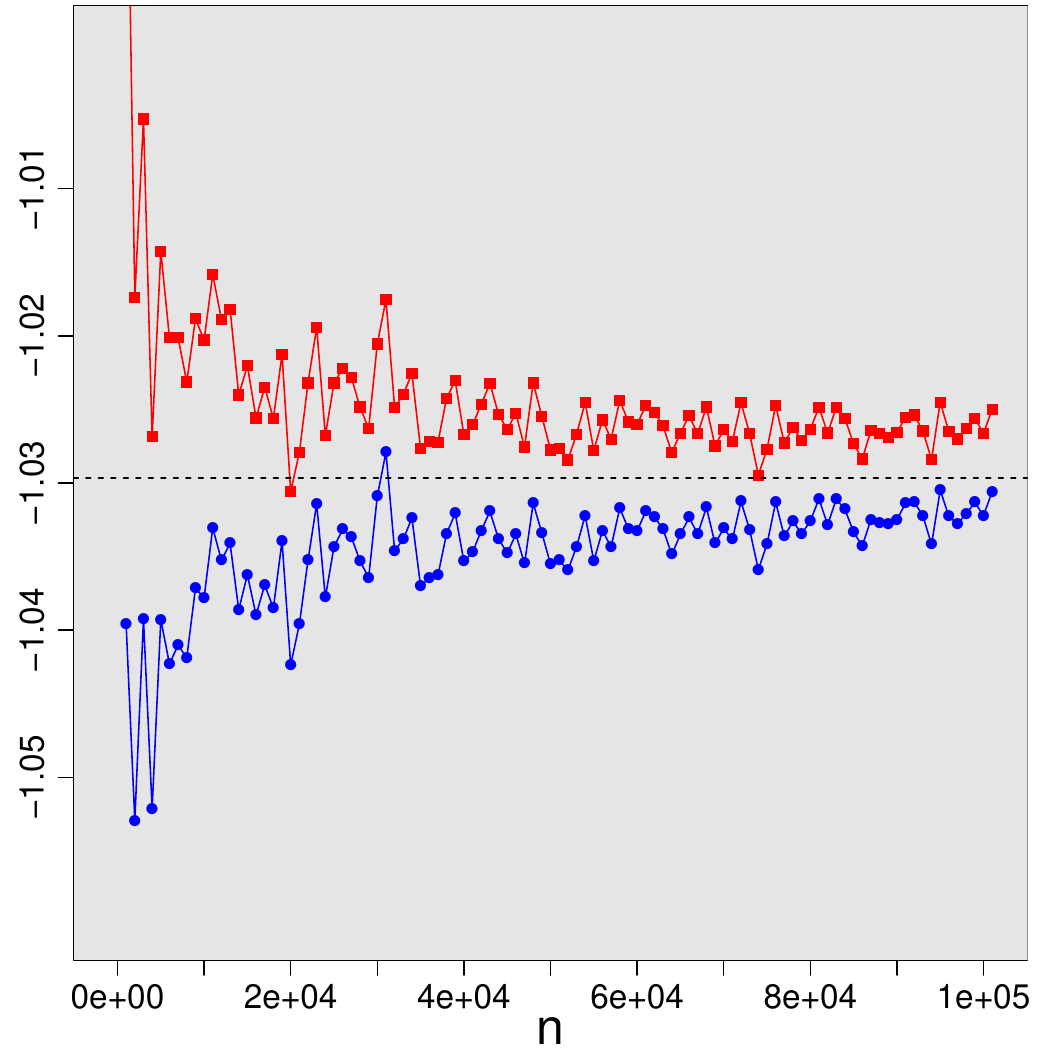}
\vspace{-0.5cm}
\caption{Confidence bounds on $\lambda_0$ as a function of $n$.}
\end{subfigure}
\quad
\begin{subfigure}[b]{0.450\textwidth}
\centering
\includegraphics[width=\textwidth, height=0.85\textwidth]{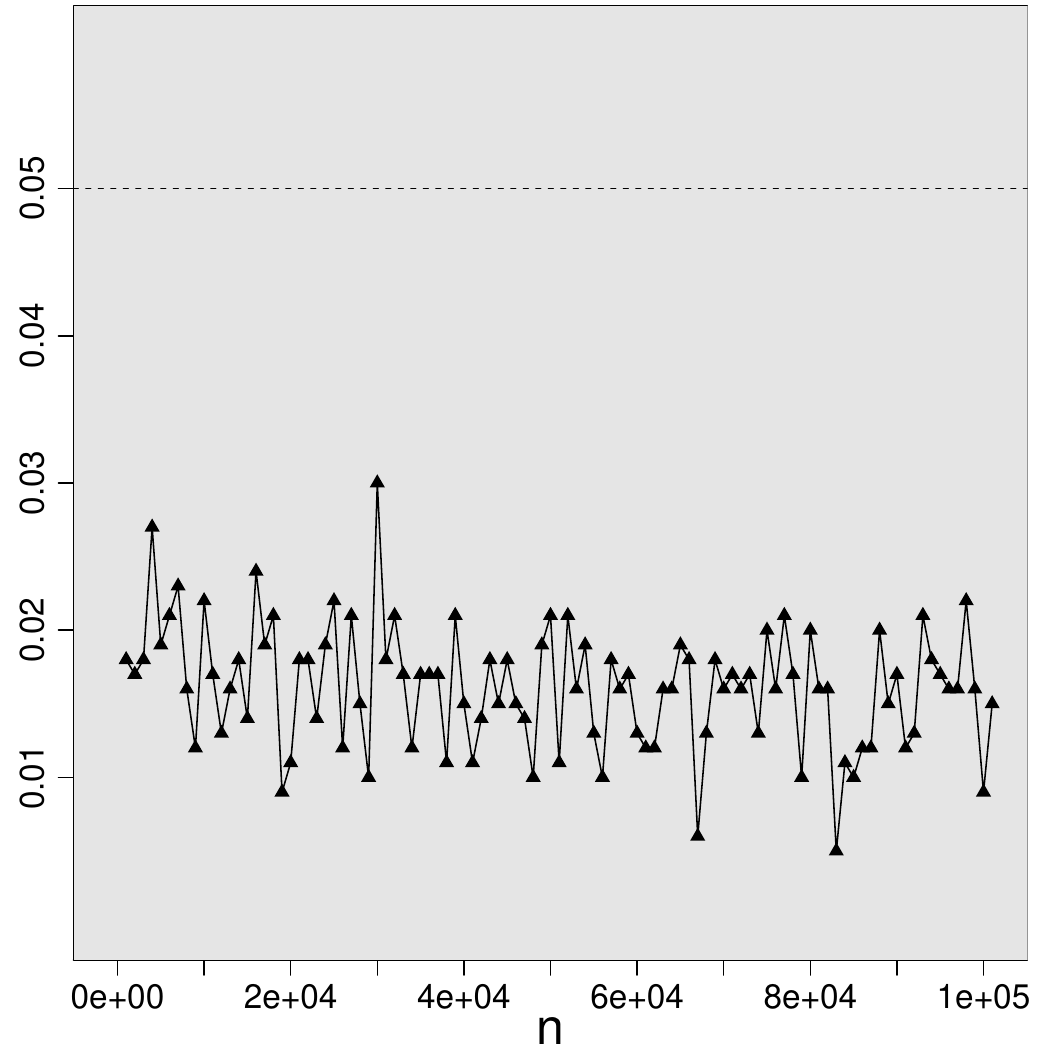}
\vspace{-0.563cm}
\caption{Empirical level as a function of $n$.}
\end{subfigure}
\caption{Example of confidences bounds and empirical levels as a function of $n$, when the parameters are set to $d = 2$, $\tau = d + 1$, $\alpha = 0.05$, $\bb{p}_0 = (0.2,0.3,0.5)^{\top}$, $\bb{v} = (1,2,3)^{\top}$, and the objective function is $f_{\bb{v},v_{d+1}}^{[1]}(\bb{p}) = \sum_{i=1}^{d+1} p_i \log p_i$.}
\label{fig:example.1}
\end{figure}

\begin{figure}[!ht]
\captionsetup{width=0.8\linewidth}
\centering
\begin{subfigure}[b]{0.450\textwidth}
\centering
\includegraphics[width=\textwidth, height=0.85\textwidth]{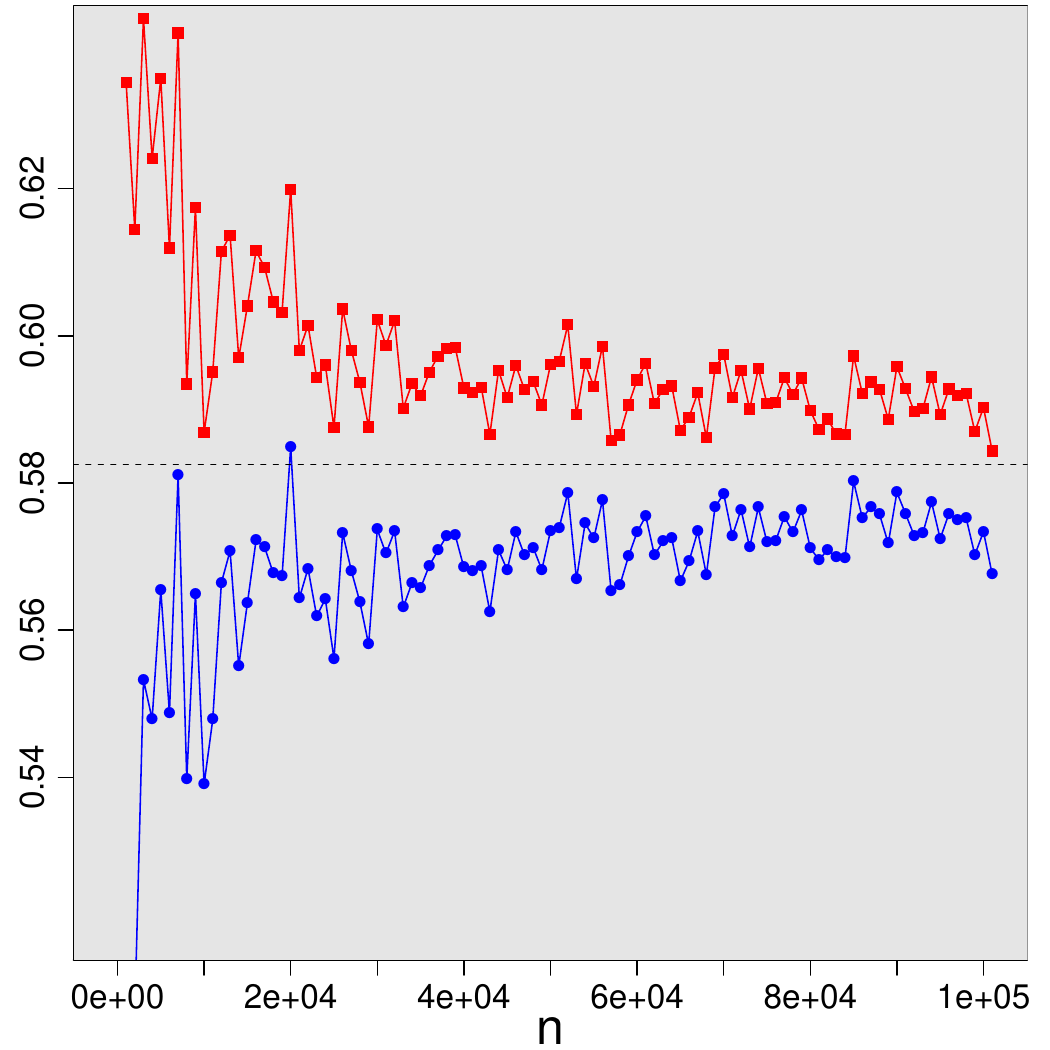}
\vspace{-0.5cm}
\caption{Confidence bounds on $\lambda_0$ as a function of $n$.}
\end{subfigure}
\quad
\begin{subfigure}[b]{0.450\textwidth}
\centering
\includegraphics[width=\textwidth, height=0.85\textwidth]{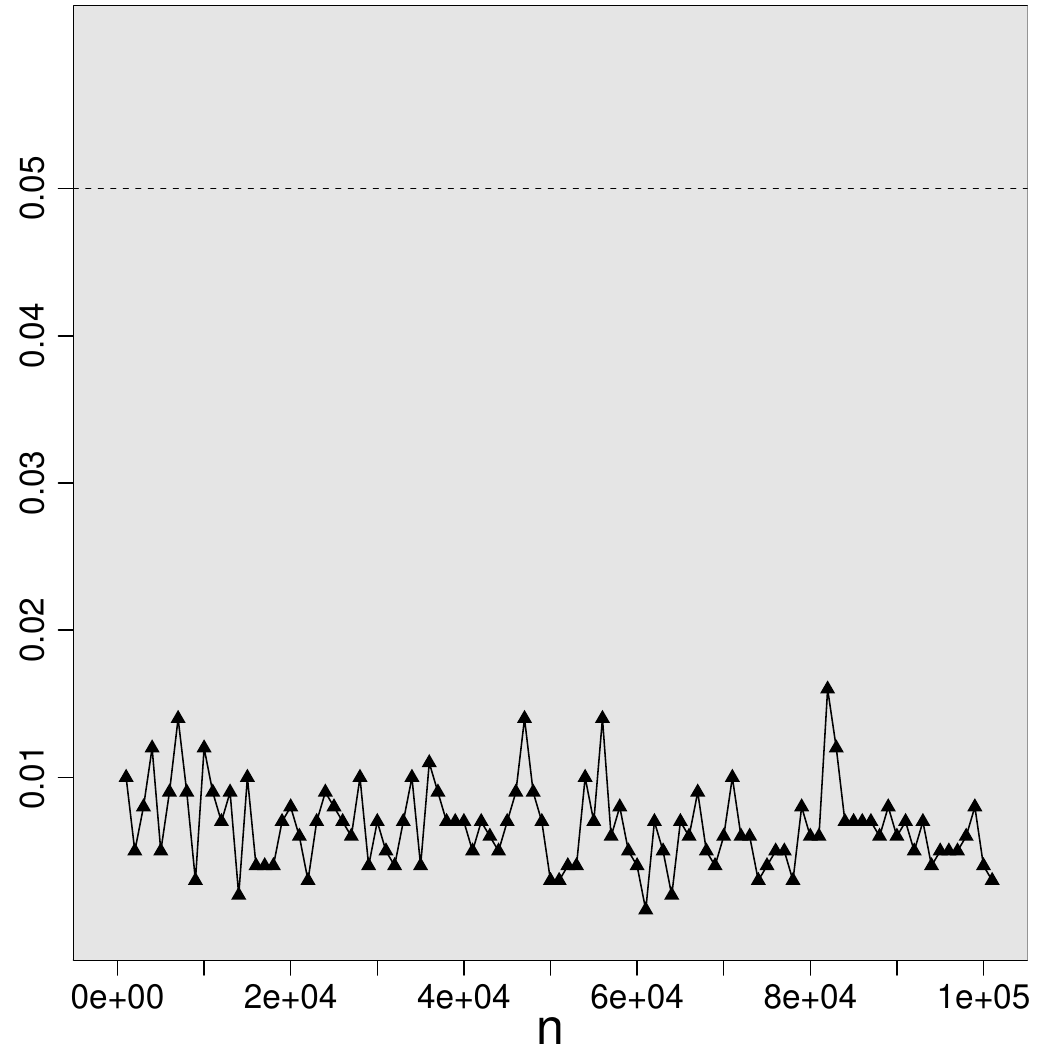}
\vspace{-0.563cm}
\caption{Empirical level as a function of $n$.}
\end{subfigure}
\caption{Example of confidences bounds and empirical levels as a function of $n$, when the parameters are set to $d = 4$, $\tau = d + 1$, $\alpha = 0.05$, $\bb{p}_0 = (0.2,0.3,0.15,0.35)^{\top}$, $\bb{v} = (1,2,3,4)^{\top}$, and the objective function is $f_{\bb{v},v_{d+1}}^{[2]}(\bb{p}) = \bb{p}^{\top} A \bb{p}$ with $A = \begin{pmatrix}\bb{a}_{\cdot 1} & \bb{a}_{\cdot 2} & \bb{a}_{\cdot 3}\end{pmatrix}$ and $\bb{a}_{\cdot 1} = (v_1 + 1, 0.5, 0.25)^{\top}$, $\bb{a}_{\cdot 2} = (0.5, v_2 + 1, 0.75)^{\top}$, $\bb{a}_{\cdot 3} = (0.25, 0.75, v_3 + 1)^{\top}$.}
\label{fig:example.2}
\end{figure}

\appendix

\begin{appendices}

\section{Proofs}\label{sec:proofs}

\begin{proof}[Proof of Proposition~\ref{prop:p.k.expansion.non.asymptotic}]
By taking the logarithm on the left-hand side of \eqref{eq:multinomial.pdf}, we have
\begin{equation}\label{eq:lem:LLT.Multinomial.eq.beginning}
\begin{aligned}
\log\left\{\frac{P_{n,\bb{p}}(\bb{k})}{n^{-d/2} \phi_{\Sigma_{\bb{p}}}(\bb{\delta}_{\bb{k},\bb{p}})}\right\} = \log\left\{\frac{n! \, (2\pi n)^{d/2}}{\prod_{i=1}^{d+1} k_i!}\right\} + \sum_{i=1}^{d+1} \left(k_i + \frac{1}{2}\right) \log p_i + \frac{1}{2} \bb{\delta}_{\bb{k},\bb{p}}^{\top} \Sigma_{\bb{p}}^{-1} \bb{\delta}_{\bb{k},\bb{p}}^{\phantom \top}.
\end{aligned}
\end{equation}
From Lindel\"of's estimate of the remainder terms in the expansion of the factorials, found for example on page~67 of \cite{MR1483074}, we know that
\begin{equation}\label{eq:factorial.expansion}
n! = \sqrt{2\pi} \exp\left\{(n + \tfrac{1}{2}) \log n - n + \lambda_n\right\}, \quad \text{where } |\lambda_n| \leq \frac{1}{12} n^{-1}.
\end{equation}
By applying \eqref{eq:factorial.expansion} in \eqref{eq:lem:LLT.Multinomial.eq.beginning} and reorganizing the terms, we get
\begin{equation}\label{eq:big.equation}
\begin{aligned}
\log\left\{\frac{P_{n,\bb{p}}(\bb{k})}{n^{-d/2} \phi_{\Sigma_{\bb{p}}}(\bb{\delta}_{\bb{k},\bb{p}})}\right\}
&= - \sum_{i=1}^{d+1} \left(k_i + \frac{1}{2}\right) \log \left(\frac{k_i}{n p_i}\right) + \frac{1}{2} \bb{\delta}_{\bb{k},\bb{p}}^{\top} \Sigma_{\bb{p}}^{-1} \bb{\delta}_{\bb{k},\bb{p}}^{\phantom \top} \\
&\quad+ \frac{1}{n} \left\{\frac{\lambda_n}{n^{-1}} - \sum_{i=1}^{d+1} \frac{\lambda_{k_i}}{k_i^{-1}} \cdot \frac{1}{p_i} \cdot \left(\frac{k_i}{n p_i}\right)^{\hspace{-0.5mm}-1}\right\}.
\end{aligned}
\end{equation}
Since $k_i / (n p_i) = 1 + \delta_{k_i,p_i} / (\sqrt{n} \, p_i)$, the above is
\begin{equation}\label{eq:big.equation.next}
\begin{aligned}
\log\left\{\frac{P_{n,\bb{p}}(\bb{k})}{n^{-d/2} \phi_{\Sigma_{\bb{p}}}(\bb{\delta}_{\bb{k},\bb{p}})}\right\}
&= - \sum_{i=1}^{d+1} n p_i \left(1 + \frac{\delta_{k_i,p_i}}{\sqrt{n} \, p_i}\right) \log \left(1 + \frac{\delta_{k_i,p_i}}{\sqrt{n} \, p_i}\right) \\
&\quad- \sum_{i=1}^{d+1} \frac{1}{2} \log \left(1 + \frac{\delta_{k_i,p_i}}{\sqrt{n} \, p_i}\right) + \frac{1}{2} \bb{\delta}_{\bb{k},\bb{p}}^{\top} \Sigma_{\bb{p}}^{-1} \bb{\delta}_{\bb{k},\bb{p}}^{\phantom \top} \\
&\quad+ \frac{1}{n} \left\{\frac{\lambda_n}{n^{-1}} - \sum_{i=1}^{d+1} \frac{\lambda_{k_i}}{k_i^{-1}} \cdot \frac{1}{p_i} \cdot \left(1 + \frac{\delta_{k_i,p_i}}{\sqrt{n} \, p_i}\right)^{\hspace{-0.5mm}-1}\right\}.
\end{aligned}
\end{equation}
Now, for $|y| \leq \tau \sqrt{(\log n) / n} \leq \sqrt{(\log n) / n^{1/2}} \leq 0.84$ (recall our assumption $n\geq \tau^4 \geq 16$), Lagrange error bounds for the following Taylor expansions imply
\begin{equation}\label{eq:Lagrange.control.Taylor.expansions}
\begin{aligned}
\left|(1 + y) \log (1 + y) - \left(y + \frac{y^2}{2} - \frac{y^3}{6}\right)\right|
&\leq \left|\frac{2}{(1 - 0.84)^3}\right| \cdot \left|\frac{y^4}{4!}\right| \leq 21 \, |y|^4, \\
\big|\log (1 + y) - y\big|
&\leq \left|\frac{-1}{(1 - 0.84)^2}\right| \cdot \left|\frac{y^2}{2!}\right| \leq 20 \, |y|^2, \\
\big|(1 + y)^{-1}\big|
&\leq \left|\frac{1}{1 - 0.84}\right| \leq 7.
\end{aligned}
\end{equation}
By applying these estimates in \eqref{eq:big.equation.next} with $y = \delta_{k_i,p_i} / (\sqrt{n} \, p_i)$, together with the error bound $|\lambda_n| \leq 1 / (12 n)$ from \eqref{eq:factorial.expansion}, we obtain
\begin{equation}\label{eq:big.equation.2}
\begin{aligned}
\log\left\{\frac{P_{n,\bb{p}}(\bb{k})}{n^{-d/2} \phi_{\Sigma_{\bb{p}}}(\bb{\delta}_{\bb{k},\bb{p}})}\right\}
&= - \sum_{i=1}^{d+1} n p_i \left\{\cancel{\frac{\delta_{k_i,p_i}}{\sqrt{n} \, p_i}} + \bcancel{\frac{1}{2} \left(\frac{\delta_{k_i,p_i}}{\sqrt{n} \, p_i}\right)^2} - \frac{1}{6} \left(\frac{\delta_{k_i,p_i}}{\sqrt{n} \, p_i}\right)^3\right\} \\
&\quad- \sum_{i=1}^{d+1} \frac{1}{2} \frac{\delta_{k_i,p_i}}{\sqrt{n} \, p_i} + \bcancel{\frac{1}{2} \bb{\delta}_{\bb{k},\bb{p}}^{\top} \Sigma_{\bb{p}}^{-1} \bb{\delta}_{\bb{k},\bb{p}}^{\phantom \top}} + R_{n,\bb{p}}(\bb{k}),
\end{aligned}
\end{equation}
where the rest $R_{n,\bb{p}}(\bb{k})$ satisfies
\begin{equation}\label{eq:bound.R.n.x.k.end.proof}
\begin{aligned}
|R_{n,\bb{p}}(\bb{k})|
&\leq \sum_{i=1}^{d+1} n p_i \cdot 21 \, \bigg|\frac{\delta_{k_i,p_i}}{\sqrt{n} \, p_i}\bigg|^4 + \sum_{i=1}^{d+1} \frac{1}{2} \cdot 20 \, \bigg|\frac{\delta_{k_i,p_i}}{\sqrt{n} \, p_i}\bigg|^2 + \frac{1}{n} \left(\frac{1}{12} + \frac{7}{12} \sum_{i=1}^{d+1} \frac{1}{p_i}\right).
\end{aligned}
\end{equation}
The cancellations in \eqref{eq:big.equation.2} come from $\delta_{k_{d+1},p_{d+1}} = - \sum_{i=1}^d \delta_{k_i,p_i}$ (recall Remark~\ref{eq:remark}) and from $(\Sigma_{\bb{p}}^{-1})_{ij} = p_i^{-1} \ind_{\{i = j\}} + p_{d+1}^{-1}$ for $1 \leq i,j \leq d$ by \eqref{eq:inverse.covariance}. This ends the proof.
\end{proof}

\begin{proof}[Proof of Proposition~\ref{prop:total.variation.bound}]
By the comparison of the total variation norm with the Hellinger distance on page 726 of \cite{MR1922539}, we already know that
\begin{equation}\label{eq:first.bound.total.variation}
\|\widetilde{\PP}_{n,\bb{p}} - \QQ_{n,\bb{p}}\| \leq \sqrt{2 \, \PP(\bb{X}\in \mathscr{B}_{n,\bb{p}}^{\hspace{0.2mm}c}) + \EE\left[\log\left\{\frac{\rd \widetilde{\PP}_{n,\bb{p}}}{\rd \QQ_{n,\bb{p}}}(\bb{X})\right\} \, \ind_{\{\bb{X}\in \mathscr{B}_{n,\bb{p}}\}}\right]}.
\end{equation}
By applying a union bound, followed by the fact that $p_i \, \tau \geq 1$ for all $\bb{p}\in \mathscr{P}_{\tau}$ and then Hoeffding's inequality, we get
\begin{align}\label{eq:concentration.bound.K}
\PP(\bb{K}\in \mathscr{B}_{n,\bb{p}}^{\hspace{0.5mm}c})
&\leq \sum_{i=1}^{d+1} \PP\left(|\delta_{K_i,p_i}| > p_i \, \tau \sqrt{\log n}\right) \leq \sum_{i=1}^{d+1} \PP\left(|\delta_{K_i,p_i}| > \sqrt{\log n}\right) \notag \\
&\leq \sum_{i=1}^{d+1} 2 \, e^{-2 (\sqrt{\log n})^2} \leq \frac{2 \, (d + 1)}{n^2},
\end{align}
and similarly, for $n\geq \tau^4  \geq 16$,
\begin{align}\label{eq:concentration.bound.X}
\PP(\bb{X}\in \mathscr{B}_{n,\bb{p}}^{\hspace{0.5mm}c})
&\leq \sum_{i=1}^{d+1} \PP\left(|\delta_{X_i,p_i}| > \sqrt{\log n}\right) \leq \sum_{i=1}^{d+1} \PP\left(|\delta_{K_i,p_i}| > \sqrt{\log n} - \tfrac{1}{2 \sqrt{n}}\right) \notag \\
&\leq \sum_{i=1}^{d+1} 2 \, e^{-2 (0.9249 \sqrt{\log n})^2} \leq \frac{2 \, (d + 1)}{n^{1.71}}.
\end{align}
To control the expectation in \eqref{eq:first.bound.total.variation}, note that if $P_{n,\bb{p}}(\bb{x})$ denotes the density function associated with $\widetilde{\PP}_{n,\bb{p}}$ (i.e., it is equal to $P_{n,\bb{p}}(\bb{k})$ whenever $\bb{k}\in \N_0^d \cap n \mathcal{S}_d^o$ is closest to $\bb{x}\in \N_0^d \cap n \mathcal{S}_d^o + [-1/2,1/2]^d$), then
\vspace{-1mm}
\begin{align}\label{eq:I.plus.II.plus.III}
&\EE\left[\log\left\{\frac{\rd \widetilde{\PP}_{n,\bb{p}}}{\rd \QQ_{n,\bb{p}}}(\bb{X})\right\} \, \ind_{\{\bb{X}\in \mathscr{B}_{n,\bb{p}}\}}\right] \notag \\
&\quad=\EE\left[\log\left\{\frac{P_{n,\bb{p}}(\bb{X})}{n^{-d/2} \phi_{\Sigma_{\bb{p}}}(\bb{\delta}_{\bb{X}\hspace{-0.3mm},\bb{p}})}\right\} \, \ind_{\{\bb{X}\in \mathscr{B}_{n,\bb{p}}\}}\right] \notag \\[1mm]
&\quad= \EE\left[\log\left\{\frac{P_{n,\bb{p}}(\bb{K})}{n^{-d/2} \phi_{\Sigma_{\bb{p}}}(\bb{\delta}_{\bb{K}\hspace{-0.3mm},\bb{p}})}\right\} \, \ind_{\{\bb{K}\in \mathscr{B}_{n,\bb{p}}\}}\right] \notag \\
&\quad\quad+ \EE\left[\log\left\{\frac{n^{-d/2} \phi_{\Sigma_{\bb{p}}}(\bb{\delta}_{\bb{K}\hspace{-0.3mm},\bb{p}})}{n^{-d/2} \phi_{\Sigma_{\bb{p}}}(\bb{\delta}_{\bb{X}\hspace{-0.3mm},\bb{p}})}\right\} \, \ind_{\{\bb{K}\in \mathscr{B}_{n,\bb{p}}\}}\right] \notag \\[1mm]
&\quad\quad+ \EE\left[\log\left\{\frac{P_{n,\bb{p}}(\bb{K})}{n^{-d/2} \phi_{\Sigma_{\bb{p}}}(\bb{\delta}_{\bb{X}\hspace{-0.3mm},\bb{p}})}\right\} \, (\ind_{\{\bb{X}\in \mathscr{B}_{n,\bb{p}}\}} - \ind_{\{\bb{K}\in \mathscr{B}_{n,\bb{p}}\}})\right] \notag \\[1mm]
&\quad\reqdef (\mathrm{I}) + (\mathrm{II}) + (\mathrm{III}).
\end{align}
By the non-asymptotic expansion in Proposition~\ref{prop:p.k.expansion.non.asymptotic}, we have
\begin{equation}\label{eq:estimate.I.begin}
\begin{aligned}
|(\mathrm{I})|
&\leq n^{-1/2} \sum_{i=1}^{d+1} \left|\EE\left\{\left(\frac{1}{6 p_i^2} \, \delta_{k_i,p_i}^3 - \frac{1}{2 p_i} \, \delta_{k_i,p_i}\right) \, \ind_{\{\bb{K}\in \mathscr{B}_{n,\bb{p}}\}}\right\}\right| \\[1mm]
&\quad+ n^{-1} \sum_{i=1}^{d+1} \left\{\frac{21}{p_i^3} \, \EE\left(|\delta_{k_i,p_i}|^4\right) + \frac{10}{p_i^2} \, \EE\left(|\delta_{k_i,p_i}|^2\right) + \frac{2 \tau}{3}\right\}.
\end{aligned}
\end{equation}
By Lemma~\ref{lem:moments.multinomial}, and our assumptions $\tau\geq d + 1\geq 2$ and $n\geq \tau^4\geq 16$, we can bound the expectations in \eqref{eq:estimate.I.begin} as follows:
\begin{align}\label{eq:estimate.I}
|(\mathrm{I})|
&\leq \frac{d + 1}{n^{1/2}} \left[\frac{\tau^2}{6} \cdot \frac{1}{\sqrt{8}} \left\{\PP(\bb{K}\in \mathscr{B}_{n,\bb{p}}^{\hspace{0.2mm}c})\right\}^{1/4} + \frac{\tau}{2} \cdot \frac{1}{2} \left\{\PP(\bb{K}\in \mathscr{B}_{n,\bb{p}}^{\hspace{0.2mm}c})\right\}^{1/2}\right] + \frac{\tau (d + 1)}{6 n} \notag \\
&\qquad+ \frac{d + 1}{n} \left\{\left(21 \tau \cdot 3 + 21 \tau^2 \cdot \frac{1}{n}\right) + 10 \tau + \frac{2 \tau}{3}\right\} \notag \\
&\leq \frac{\tau^2 (d+1)}{n^{1/2}} \left\{\frac{1}{6\sqrt{8}} \sqrt[4]{\frac{2 \, (d + 1)}{n^2}} + \frac{1}{4 \tau} \sqrt{\frac{2 \, (d + 1)}{n^2}}\right\} + \frac{38.23 \, \tau^2 \, (d + 1)}{n} \notag \\
&\leq \frac{32.29 \, \tau^2 \, (d + 1)^{5/4}}{n}.
\end{align}

For the term $(\mathrm{II})$ in \eqref{eq:I.plus.II.plus.III}, note that
\begin{align}\label{eq:estimate.II.before}
&\log\left\{\frac{n^{-d/2} \phi_{\Sigma_{\bb{p}}}(\bb{\delta}_{\bb{K}\hspace{-0.3mm},\bb{p}})}{n^{-d/2} \phi_{\Sigma_{\bb{p}}}(\bb{\delta}_{\bb{X}\hspace{-0.3mm},\bb{p}})}\right\} \notag \\
&\quad= \frac{1}{2n} (\bb{X} - n \bb{p})^{\top} \Sigma_{\bb{p}}^{-1} (\bb{X} - n \bb{p}) - \frac{1}{2n} (\bb{K} - n \bb{p})^{\top} \Sigma_{\bb{p}}^{-1} (\bb{K} - n \bb{p}) \notag \\[1mm]
&\quad= \frac{1}{2n} (\bb{X} - \bb{K})^{\top} \Sigma_{\bb{p}}^{-1} (\bb{X} - \bb{K}) \notag \\
&\quad\quad+ \frac{1}{2n} \left\{(\bb{X} - \bb{K})^{\top} \Sigma_{\bb{p}}^{-1} (\bb{K} - n \bb{p}) + (\bb{K} - n \bb{p})^{\top} \Sigma_{\bb{p}}^{-1} (\bb{X} - \bb{K})\right\}.
\end{align}
Using the expression $(\Sigma_{\bb{p}}^{-1})_{ij} = p_i^{-1} \ind_{\{i = j\}} + p_{d+1}^{-1}$ in \eqref{eq:inverse.covariance}, the independence of the random variables $(X_i - K_i)_{i=1}^d$, and our assumptions $\tau\geq d + 1\geq 2$ and $n\geq \tau^4\geq 16$, we obtain
\begin{align}\label{eq:estimate.II}
|(\mathrm{II})|
&\leq \frac{1}{2n} \sum_{i=1}^d \frac{(p_i^{-1} + p_{d+1}^{-1})}{12} + \frac{1}{\sqrt{n}} \sum_{i=1}^d \frac{\left(\frac{1}{p_i} + \frac{d}{p_{d+1}}\right)}{2} \, \sqrt{\EE\left(|\delta_{k_i,p_i}|^2\right)} \sqrt{\PP(\bb{K}\in \mathscr{B}_{n,\bb{p}}^{\hspace{0.2mm}c})} \notag \\
&\leq \frac{1}{2n} \cdot \frac{\tau (d + 1)}{6} + \frac{1}{\sqrt{n}} \cdot \frac{\tau (d + 1)^2}{2} \cdot \sqrt{\frac{1}{4}} \cdot \sqrt{\frac{2 \, (d + 1)}{n^2}} \notag \\[1mm]
&\leq \frac{\tau (d + 1)}{3 n}.
\end{align}

To bound the term $(\mathrm{III})$ in \eqref{eq:I.plus.II.plus.III}, we first derive a rough bound for the log-ratio in \eqref{eq:LLT.order.2.log} on the set $\{\bb{K}\in \mathscr{B}_{n,\bb{p}}\} \bigtriangleup \{\bb{X}\in \mathscr{B}_{n,\bb{p}}\}$. Note that the three coefficients $(21,20,7)$ we found in \eqref{eq:Lagrange.control.Taylor.expansions}, and hence those in \eqref{eq:bound.R.n.x.k.end.proof}, can be replaced by $(84,50,10)$ on the event $\{\bb{K}\in \mathscr{B}_{n,\bb{p}}^{\hspace{0.2mm}c}\} \cap \{\bb{X}\in \mathscr{B}_{n,\bb{p}}\}$ by rerunning the proof of Proposition~\ref{prop:p.k.expansion.non.asymptotic}, since the bound on $|y|$ just above \eqref{eq:Lagrange.control.Taylor.expansions} in that case becomes
\begin{equation}
\max_{1 \leq i \leq d + 1} \left|\frac{\delta_{k_i,p_i}}{\sqrt{n} \, p_i}\right| \leq \tau \sqrt{\frac{\log n}{n}} + \frac{\tau}{2 n} \leq 1.076 \tau \sqrt{\frac{\log n}{n}} \leq 0.9, \quad \text{when } n\geq \tau^4\geq 16.
\end{equation}
Taking the above into account, and using our assumptions $\tau\geq d + 1\geq 2$ and $n\geq \tau^4\geq 16$, we have
\begin{align}\label{eq:bound.III.begin.1}
&\left|\log\left\{\frac{P_{n,\bb{p}}(\bb{K})}{n^{-d/2} \phi_{\Sigma_{\bb{p}}}(\bb{\delta}_{\bb{K}\hspace{-0.3mm},\bb{p}})}\right\} \, \ind_{\{\bb{K}\in \mathscr{B}_{n,\bb{p}}\} \bigtriangleup \{\bb{X}\in \mathscr{B}_{n,\bb{p}}\}}\right| \notag \\
&\qquad\leq \sum_{i=1}^{d+1} \left(\frac{n p_i}{6} \bigg|\frac{\delta_{k_i,p_i}}{\sqrt{n} \, p_i}\bigg|^3 + \frac{1}{2} \bigg|\frac{\delta_{k_i,p_i}}{\sqrt{n} \, p_i}\bigg|\right) \notag \\
&\qquad\quad+ \sum_{i=1}^{d+1} n p_i \cdot 84 \, \bigg|\frac{\delta_{k_i,p_i}}{\sqrt{n} \, p_i}\bigg|^4 + \sum_{i=1}^{d+1} \frac{1}{2} \cdot 50 \, \bigg|\frac{\delta_{k_i,p_i}}{\sqrt{n} \, p_i}\bigg|^2 + \frac{11 \tau (d + 1)}{12 n} \notag \\
&\qquad\leq \frac{n}{6} \left(1.076 \, \tau \sqrt{\frac{\log n}{n}}\right)^3 + \frac{d + 1}{2} \left(1.076 \, \tau \sqrt{\frac{\log n}{n}}\right) \notag \\
&\qquad\quad+ 84 n \left(1.076 \, \tau \sqrt{\frac{\log n}{n}}\right)^4 + 25 \left(1.076 \, \tau \sqrt{\frac{\log n}{n}}\right)^2 + \frac{11 \tau (d + 1)}{12 n} \notag \\
&\qquad\leq \frac{96.25 \, \tau^3 (\log n)^{3/2}}{n^{1/2}}.
\end{align}
We also have the following rough bound from \eqref{eq:estimate.II.before},
\begin{align}\label{eq:bound.III.begin.2}
&\left|\log\left\{\frac{n^{-d/2} \phi_{\Sigma_{\bb{p}}}(\bb{\delta}_{\bb{K}\hspace{-0.3mm},\bb{p}})}{n^{-d/2} \phi_{\Sigma_{\bb{p}}}(\bb{\delta}_{\bb{X}\hspace{-0.3mm},\bb{p}})}\right\} \, \ind_{\{\bb{K}\in \mathscr{B}_{n,\bb{p}}\} \bigtriangleup \{\bb{X}\in \mathscr{B}_{n,\bb{p}}\}}\right| \notag \\
&\quad\leq \frac{1}{2n} \sum_{i=1}^d \frac{\left(\frac{1}{p_i} + \frac{d}{p_{d+1}}\right)}{4} + \frac{1}{\sqrt{n}} \sum_{i=1}^d \frac{\left(\frac{1}{p_i} + \frac{d}{p_{d+1}}\right)}{2} \, p_i \bigg|\frac{\delta_{k_i,p_i}}{\sqrt{n} \, p_i}\bigg| \, \ind_{\{\bb{K}\in \mathscr{B}_{n,\bb{p}}\} \bigtriangleup \{\bb{X}\in \mathscr{B}_{n,\bb{p}}\}} \notag \\
&\quad\leq \frac{1}{2n} \cdot \frac{\tau (d + 1)^2}{4} + \frac{1}{\sqrt{n}} \cdot \frac{\tau (d + 1)}{2} \cdot 1.076 \, \tau \sqrt{\frac{\log n}{n}} \leq \frac{0.62 \, \tau^3 (\log n)^{1/2}}{n},
\end{align}
again using our assumptions $\tau\geq d + 1\geq 2$ and $n\geq \tau^4\geq 16$. Putting the rough bounds \eqref{eq:bound.III.begin.1} and \eqref{eq:bound.III.begin.2} together yields
\begin{align}\label{eq:estimate.III}
|(\mathrm{III})|
&\leq \left|\EE\left(\left[\log\left\{\frac{P_{n,\bb{p}}(\bb{K})}{n^{-d/2} \phi_{\Sigma_{\bb{p}}}(\bb{\delta}_{\bb{K}\hspace{-0.3mm},\bb{p}})}\right\} + \log\left\{\frac{n^{-d/2} \phi_{\Sigma_{\bb{p}}}(\bb{\delta}_{\bb{K}\hspace{-0.3mm},\bb{p}})}{n^{-d/2} \phi_{\Sigma_{\bb{p}}}(\bb{\delta}_{\bb{X}\hspace{-0.3mm},\bb{p}})}\right\}\right] \, \ind_{\{\bb{K}\in \mathscr{B}_{n,\bb{p}}\} \bigtriangleup \{\bb{X}\in \mathscr{B}_{n,\bb{p}}\}}\right)\right| \notag \\
&\leq \left\{\frac{96.25 \, \tau^3 (\log n)^{3/2}}{n^{1/2}} + \frac{0.62 \, \tau^3 (\log n)^{1/2}}{n}\right\} \, \EE\left[|\ind_{\{\bb{K}\in \mathscr{B}_{n,\bb{p}}\} \bigtriangleup \{\bb{X}\in \mathscr{B}_{n,\bb{p}}\}}|\right] \notag \\
&\leq \frac{96.31 \, \tau^3 (\log n)^{3/2}}{n^{1/2}} \, \EE\left[|\ind_{\{\bb{K}\in \mathscr{B}_{n,\bb{p}}\} \bigtriangleup \{\bb{X}\in \mathscr{B}_{n,\bb{p}}\}}|\right].
\end{align}

Putting \eqref{eq:estimate.I}, \eqref{eq:estimate.II} and \eqref{eq:estimate.III} in \eqref{eq:I.plus.II.plus.III}, together with the bound
\begin{align}\label{eq:III.before.combining}
\EE\left(|\ind_{\{\bb{X}\in \mathscr{B}_{n,\bb{p}}\}} - \ind_{\{\bb{K}\in \mathscr{B}_{n,\bb{p}}\}}|\right)
&=\EE\left(|\ind_{\{\bb{X}\in \mathscr{B}_{n,\bb{p}}, \bb{K}\in \mathscr{B}_{n,\bb{p}}^{\hspace{0.2mm}c}\}} - \ind_{\{\bb{X}\in \mathscr{B}_{n,\bb{p}}^{\hspace{0.2mm}c}, \bb{K}\in \mathscr{B}_{n,\bb{p}}\}}|\right) \notag \\[1mm]
&\leq \PP(\bb{K}\in \mathscr{B}_{n,\bb{p}}^{\hspace{0.2mm}c}) + \PP(\bb{X}\in \mathscr{B}_{n,\bb{p}}^{\hspace{0.2mm}c}) \notag \\[0.5mm]
&\leq \frac{2 \, (d + 1)}{n^2} + \frac{2 \, (d + 1)}{n^{1.71}} \leq \frac{2.90 \, (d + 1)}{n^{1.71}},
\end{align}
(recall \eqref{eq:concentration.bound.K} and \eqref{eq:concentration.bound.X}) yields
\begin{align}\label{eq:I.plus.II.plus.III.end}
\left|\EE\left[\log\left\{\frac{\rd \widetilde{\PP}_{n,\bb{p}}}{\rd \QQ_{n,\bb{p}}}(\bb{X})\right\} \, \ind_{\{\bb{X}\in \mathscr{B}_{n,\bb{p}}\}}\right]\right|
&\leq |(\mathrm{I})| + |(\mathrm{II})| + |(\mathrm{III})| \notag \\[-1mm]
&\leq \frac{32.29 \, \tau^2 \, (d + 1)^{5/4}}{n} + \frac{\tau (d + 1)}{3 n} \notag \\
&\quad+ \frac{96.31 \, \tau^3 (\log n)^{3/2}}{n^{1/2}} \cdot \frac{2.90 \, (d + 1)}{n^{1.71}} \notag \\[1mm]
&\leq \frac{64.31 \, \tau^3 \, (d + 1)}{n}.
\end{align}
(Again, we used the fact that $\tau\geq d + 1\geq 2$ and $n\geq \tau^4\geq 16$ to obtain the last inequality.) Now, putting \eqref{eq:concentration.bound.X} and \eqref{eq:I.plus.II.plus.III.end} together in \eqref{eq:first.bound.total.variation} gives the conclusion.
\end{proof}

\begin{proof}[Proof of Theorem~\ref{thm.cdf.comparison}]
The case $d = 1$ can easily be treated separately using the Berry-Esseen theorem (in fact the bound in \eqref{eq:thm.cdf.comparison} would be of the form $C n^{-1/2}$ in that case). Alternatively, we can prove our statement as we do below, for all $d\geq 1$ and any $\ell > 0$, only assuming further that $\ell > 1/2$ when $d = 1$.

By the triangle inequality, we have
\begin{equation}\label{eq:diff.g.h.begin}
\begin{aligned}
&\left|\PP\left(\bb{\delta}_{\bb{K}\hspace{-0.5mm},\bb{p}}^{\top} \Sigma_{\bb{p}}^{-1} \bb{\delta}_{\bb{K}\hspace{-0.5mm},\bb{p}}^{\phantom{\top}} \leq \ell\right) - \PP\left(\bb{\delta}_{\bb{Y}\hspace{-0.5mm},\bb{p}}^{\top} \Sigma_{\bb{p}}^{-1} \bb{\delta}_{\bb{Y}\hspace{-0.5mm},\bb{p}}^{\phantom{\top}} \leq \ell\right)\right| \\[0.5mm]
&\quad\leq \PP_{n,\bb{p}}(\mathscr{B}_{n,\bb{p}}^{\hspace{0.5mm}c}) + \left|\hspace{-1mm}
\begin{array}{l}
\PP_{n,\bb{p}}\left(\left\{\bb{k}\in \mathscr{B}_{n,\bb{p}} : \bb{\delta}_{\bb{k},\bb{p}}^{\top} \Sigma_{\bb{p}}^{-1} \bb{\delta}_{\bb{k},\bb{p}}^{\phantom{\top}} \leq \ell\right\}\right) \\[0.5mm]
- \widetilde{\PP}_{n\hspace{-0.5mm},\bb{p}}\left(\left\{\bb{k}\in \mathscr{B}_{n,\bb{p}} : \bb{\delta}_{\bb{k},\bb{p}}^{\top} \Sigma_{\bb{p}}^{-1} \bb{\delta}_{\bb{k},\bb{p}}^{\phantom{\top}} \leq \ell\right\} + \big[-\tfrac{1}{2},\tfrac{1}{2}\big]^d\right)
\end{array}
\hspace{-1mm}\right| \\
&\quad\qquad+ |\widetilde{\PP}_{n\hspace{-0.5mm},\bb{p}} - \QQ_{n,\bb{p}}|\left(\big\{\bb{k}\in \mathscr{B}_{n,\bb{p}} : \bb{\delta}_{\bb{k},\bb{p}}^{\top} \Sigma_{\bb{p}}^{-1} \bb{\delta}_{\bb{k},\bb{p}}^{\phantom{\top}} \leq \ell\big\} + \big[-\tfrac{1}{2},\tfrac{1}{2}\big]^d\right) \\
&\quad\qquad+ \QQ_{n,\bb{p}}\left[\hspace{-1mm}
\begin{array}{l}
\left(\left\{\bb{k}\in \mathscr{B}_{n,\bb{p}} : \bb{\delta}_{\bb{k},\bb{p}}^{\top} \Sigma_{\bb{p}}^{-1} \bb{\delta}_{\bb{k},\bb{p}}^{\phantom{\top}} \leq \ell\right\} + \big[-\tfrac{1}{2},\tfrac{1}{2}\big]^d\right) \\
\bigtriangleup \, \left\{\bb{y}\in \big(\mathscr{B}_{n,\bb{p}} + [-\tfrac{1}{2},\tfrac{1}{2}]^d\big) : \bb{\delta}_{\bb{y},\bb{p}}^{\top} \Sigma_{\bb{p}}^{-1} \bb{\delta}_{\bb{y},\bb{p}}^{\phantom{\top}} \leq \ell\right\}
\end{array}
\hspace{-1mm}\right] \\[0.5mm]
&\quad\qquad+ \QQ_{n,\bb{p}}\left[\hspace{-1mm}
\begin{array}{l}
\left\{\bb{y}\in \big(\mathscr{B}_{n,\bb{p}} + [-\tfrac{1}{2},\tfrac{1}{2}]^d\big) : \bb{\delta}_{\bb{y},\bb{p}}^{\top} \Sigma_{\bb{p}}^{-1} \bb{\delta}_{\bb{y},\bb{p}}^{\phantom{\top}} \leq \ell\right\} \\[1mm]
\bigtriangleup \, \left\{\bb{y}\in \R^d : \bb{\delta}_{\bb{y},\bb{p}}^{\top} \Sigma_{\bb{p}}^{-1} \bb{\delta}_{\bb{y},\bb{p}}^{\phantom{\top}} \leq \ell\right\}
\end{array}
\hspace{-1mm}\right] \\[1mm]
&\quad\reqdef (A)_{\ell} + (B)_{\ell} + (C)_{\ell} + (D)_{\ell} + (E)_{\ell}.
\end{aligned}
\end{equation}
First, as in \eqref{eq:concentration.bound.K}, we have
\begin{equation}\label{eq:bound.A}
(A)_{\ell} = \PP(\bb{K}\in \mathscr{B}_{n,\bb{p}}^{\hspace{0.5mm}c}) \leq \frac{2 \, (d + 1)}{n^2}.
\end{equation}
Second, note that by the definition of $\widetilde{\PP}_{n\hspace{-0.5mm},\bb{p}}$,
\begin{equation}\label{eq:bound.B}
(B)_{\ell} = 0.
\end{equation}
Third, to bound $(C)_{\ell}$, we can use the total variation bound in Proposition~\ref{prop:total.variation.bound}. We have
\begin{equation}\label{eq:bound.C}
(C)_{\ell} \leq \frac{8.03 \, \tau^{3/2} (d + 1)^{1/2}}{n^{1/2}}.
\end{equation}

To bound $(D)_{\ell}$, note that $\left\{\bb{y}\in \R^d : n^{-d/2} \phi_{\Sigma_{\bb{p}}}(\bb{\delta}_{\bb{y}}) > n^{-d/2} \phi_{\Sigma_{\bb{p}}}(\bb{\delta}_{\bb{n}})\right\}$ is a convex and connected set. Hence, if we ignore the restriction to $\mathscr{B}_{n,\bb{p}} + [-1/2,1/2]^d$, the symmetric difference
\begin{equation}\label{eq:symmetric.difference.bound.E}
\begin{aligned}
&\left(\left\{\bb{k}\in \mathscr{B}_{n,\bb{p}} : \bb{\delta}_{\bb{k},\bb{p}}^{\top} \Sigma_{\bb{p}}^{-1} \bb{\delta}_{\bb{k},\bb{p}}^{\phantom{\top}} \leq \ell\right\} + \big[-\tfrac{1}{2},\tfrac{1}{2}\big]^d\right) \\
&\bigtriangleup \, \left\{\bb{y}\in \big(\mathscr{B}_{n,\bb{p}} + [-\tfrac{1}{2},\tfrac{1}{2}]^d\big) : \bb{\delta}_{\bb{y},\bb{p}}^{\top} \Sigma_{\bb{p}}^{-1} \bb{\delta}_{\bb{y},\bb{p}}^{\phantom{\top}} \leq \ell\right\},
\end{aligned}
\end{equation}
is contained in a $d$-dimensional shell of thickness at most $1$ in any given direction that is parallel to one of the $d$ axes. Figure~\ref{fig:ellipse.discrete.versus.continuous} illustrates the situation when $d = 2$.
\begin{figure}
\captionsetup{width=0.7\textwidth}
\centering
\includegraphics[width=70mm]{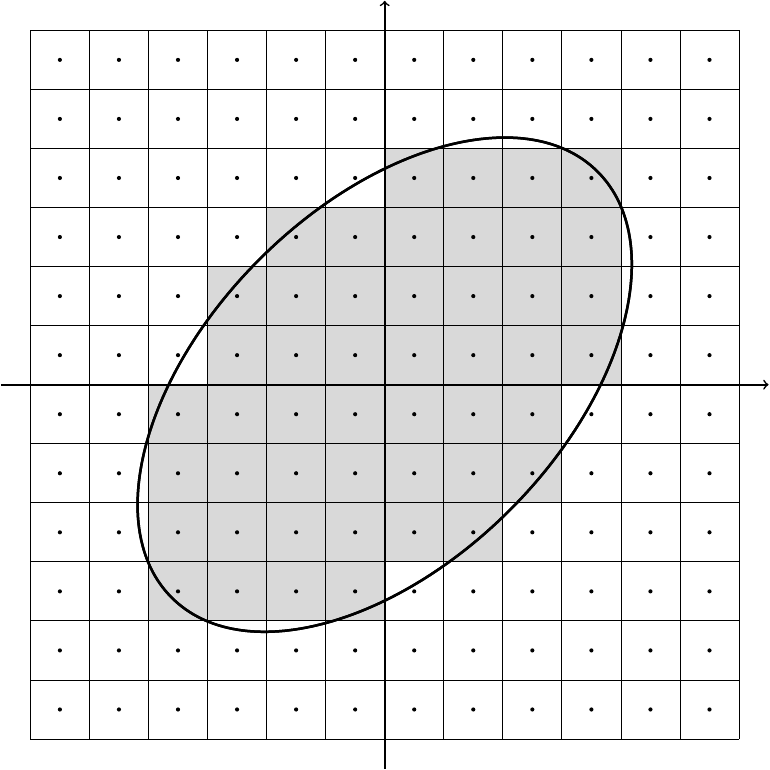}
\caption{The symmetric difference \eqref{eq:symmetric.difference.bound.E} is given by the difference between the gray region and the region that the ellipse covers. The ellipse itself represents the contour line of the multivariate normal  $\mathrm{Normal}_d(n \bb{p}, n \Sigma_{\bb{p}})$ at height $e^{-\ell/2} / \sqrt{(2\pi n)^d \det(\Sigma_{\bb{p}})}$.}
\label{fig:ellipse.discrete.versus.continuous}
\end{figure}
When $\bb{k}\in \mathscr{B}_{n,\bb{p}}$, we have
\begin{align}
&\left|\frac{1}{2} \bb{\delta}_{\bb{k},\bb{p}}^{\top} \Sigma_{\bb{p}}^{-1} \frac{(1,\dots,1)}{\sqrt{n}}\right| = \frac{\tau}{2} \sqrt{\frac{\log n}{n}} \sum_{i=1}^d (p_i^{-1} + d \, p_{d+1}^{-1}) \leq \frac{\tau^2 (d + 1)^2 (\log n)^{1/2}}{2 n^{1/2}}, \\
&\left|\frac{1}{2} \frac{(1,\dots,1)^{\top}}{\sqrt{n}} \Sigma_{\bb{p}}^{-1} \bb{\delta}_{\bb{k},\bb{p}}\right| = \frac{\tau}{2} \sqrt{\frac{\log n}{n}} \sum_{i=1}^d (p_i^{-1} + d \, p_{d+1}^{-1}) \leq \frac{\tau^2 (d + 1)^2 (\log n)^{1/2}}{2 n^{1/2}}, \\
&\left|\frac{1}{2} \frac{(1,\dots,1)}{\sqrt{n}}^{\top} \Sigma_{\bb{p}}^{-1} \frac{(1,\dots,1)}{\sqrt{n}}\right| = \frac{1}{2 n} \sum_{i=1}^d (p_i^{-1} + d \, p_{d+1}^{-1}) \leq \frac{\tau (d + 1)^2}{2 n},
\end{align}
so that
\begin{align}\label{eq:bound.diff.ellipses}
\left|\frac{1}{2} \Big(\bb{\delta}_{\bb{k},\bb{p}} \pm \frac{(1,\dots,1)}{\sqrt{n}}\Big)^{\top} \Sigma_{\bb{p}}^{-1} \Big(\bb{\delta}_{\bb{k},\bb{p}} \pm \frac{(1,\dots,1)}{\sqrt{n}}\Big) - \frac{1}{2} \bb{\delta}_{\bb{k},\bb{p}} \Sigma_{\bb{p}}^{-1} \bb{\delta}_{\bb{k},\bb{p}}\right| \leq b_n,
\end{align}
(using $\tau\geq d+1\geq 2$ and $n\geq \tau^4\geq 16$) where
\begin{equation}
b_n \leqdef \frac{1.05 \, \tau^2 (d + 1)^2 (\log n)^{1/2}}{n^{1/2}}.
\end{equation}
Using Equation~\eqref{eq:bound.diff.ellipses}, Remark~\ref{rem:gamma.integral} and the mean value theorem (only when $d\geq 2$), we deduce
\vspace{-2mm}
\begin{align}\label{eq:bound.E.begin}
(D)_{\ell}
&\leq \PP\left(\bb{\delta}_{\bb{Y}\hspace{-0.5mm},\bb{p}}^{\top} \Sigma_{\bb{p}}^{-1} \bb{\delta}_{\bb{Y}\hspace{-0.5mm},\bb{p}}^{\phantom{\top}} \leq \ell + b_n\right) - \PP\left\{\bb{\delta}_{\bb{Y}\hspace{-0.5mm},\bb{p}}^{\top} \Sigma_{\bb{p}}^{-1} \bb{\delta}_{\bb{Y}\hspace{-0.5mm},\bb{p}}^{\phantom{\top}} \leq \max(0,\ell - b_n)\right\} \notag \\[2mm]
&= \overline{\gamma}\left\{d/2,(\ell+b_n)/2\right\} - \overline{\gamma}\left\{d/2,\max(0,\ell-b_n)/2\right\} \\
&\leq \max_{t\in [0,\infty) : |2t - \ell| \leq b_n} \frac{t^{d/2-1} e^{-t}}{\Gamma(d/2)} \cdot b_n
\end{align}
Now, notice that the maximum above is attained at $t = \max(0,\ell/2 - b_n/2)$ for $d\in \{1,2\}$. More generally, the function $t\mapsto t^{d/2-1} e^{-t}$ maximizes on $(0,\infty)$ at $t = d/2 - 1$ for all $d\geq 3$, and
\begin{equation}
\Gamma(d/2) = (d/2 - 1) \Gamma(d/2 - 1) \geq \sqrt{2\pi} (d/2 - 1)^{d/2 + 1/2} e^{-(d/2 - 1)}, \quad \text{for all } d \geq 3,
\end{equation}
by a standard Stirling lower bound (see, e.g., Lemma~1 of \cite{MR162751}, or Theorem~2.2 of \cite{MR3684463} for a more precise bound). Using this in \eqref{eq:bound.E.begin} shows that, for any given $\ell > 0$ (we assumed that $\ell > 1/2$ when $d = 1$),
\begin{equation}\label{eq:bound.D}
(D)_{\ell} \leq
\left\{\hspace{-1mm}
\begin{array}{ll}
\frac{\{\max(0,\ell/2 - b_n/2)\}^{-1/2}}{\sqrt{\pi} \, e^{\max(0,\ell/2 - b_n/2)}} \cdot b_n, &\mbox{if } d = 1 \\
e^{-\max(0,\ell/2 - b_n/2)} \cdot b_n, &\mbox{if } d = 2 \\[1mm]
\frac{1}{\sqrt{2\pi} (d/2  - 1)} \cdot b_n, &\mbox{if } d\geq 3
\end{array}
\hspace{-1mm}\right\} \leq b_n = \frac{1.05 \, \tau^2 (d + 1)^2 (\log n)^{1/2}}{n^{1/2}}.
\end{equation}

To bound $(E)_{\ell}$, we combine \eqref{eq:concentration.bound.X} and Proposition~\ref{prop:total.variation.bound} to obtain
\begin{align}\label{eq:bound.E}
(E)_{\ell}
&\leq \QQ_{n,\bb{p}}\big\{\big(\mathscr{B}_{n,\bb{p}} + [-\tfrac{1}{2},\tfrac{1}{2}]^d\big)^c\big\} \notag \\[1.5mm]
&\leq \PP(\bb{X}\in \mathscr{B}_{n,\bb{p}}^{\hspace{0.5mm}c}) + |\widetilde{\PP}_{n\hspace{-0.5mm},\bb{p}} - \QQ_{n,\bb{p}}|\big\{\big(\mathscr{B}_{n,\bb{p}} + [-\tfrac{1}{2},\tfrac{1}{2}]^d\big)^c\big\} \notag \\[1mm]
&\leq \frac{2 \, (d + 1)}{n^{1.71}} + \frac{8.03 \, \tau^{3/2} (d + 1)^{1/2}}{n^{1/2}} \notag \\
&\leq \frac{8.07 \, \tau^{3/2} (d + 1)^{1/2}}{n^{1/2}}.
\end{align}

Finally, by applying the bounds we found for $(A)_{\ell}$, $(B)_{\ell}$, $(C)_{\ell}$, $(D)_{\ell}$ and $(E)_{\ell}$ in \eqref{eq:bound.A}, \eqref{eq:bound.B}, \eqref{eq:bound.C}, \eqref{eq:bound.D} and \eqref{eq:bound.E}, respectively, together in \eqref{eq:diff.g.h.begin}, we get
\begin{align}
&\left|\PP\left(\bb{\delta}_{\bb{K}\hspace{-0.5mm},\bb{p}}^{\top} \Sigma_{\bb{p}}^{-1} \bb{\delta}_{\bb{K}\hspace{-0.5mm},\bb{p}}^{\phantom{\top}} \leq \ell\right) - \PP\left(\bb{\delta}_{\bb{Y}\hspace{-0.5mm},\bb{p}}^{\top} \Sigma_{\bb{p}}^{-1} \bb{\delta}_{\bb{Y}\hspace{-0.5mm},\bb{p}}^{\phantom{\top}} \leq \ell\right)\right| \notag \\[1.5mm]
&\quad\leq (A)_{\ell} + (B)_{\ell} + (C)_{\ell} + (C)_{\ell} + (D)_{\ell} + (E)_{\ell} \notag \\[0.5mm]
&\quad\leq \frac{2 \, (d + 1)}{n^2} + 0 + \frac{8.03 \, \tau^{3/2} (d + 1)^{1/2}}{n^{1/2}} + \frac{1.05 \, \tau^2 (d + 1)^2 (\log n)^{1/2}}{n^{1/2}} + \frac{8.07 \, \tau^{3/2} (d + 1)^{1/2}}{n^{1/2}} \notag \\[0.5mm]
&\quad\leq \frac{1.26 \, \tau^3 (d + 1) (\log n)^{3/2}}{n^{1/2}}.
\end{align}
This ends the proof.
\end{proof}

\begin{proof}[Proof of Corollary~\ref{cor:multivariate.Tusnady.inequality}]
By Theorem~\ref{thm.cdf.comparison}, we know that, for $\tau \geq d + 1$ and $n\geq \tau^4$,
\begin{equation}\label{eq:thm:multivariate.Tusnady.inequality.eq.begin}
G(\ell) \leq F(\ell) + \frac{1.26 \, \tau^3 (d + 1) (\log n)^{3/2}}{n^{1/2}}.
\end{equation}
By the mean value theorem, we have
\begin{equation}
\begin{aligned}
&\frac{1.26 \, \tau^3 (d + 1) (\log n)^{3/2}}{n^{1/2}} \cdot \frac{G'(\ell^{\star})}{G'(\ell)} \leq G(\ell) - G(c(\ell)),
\end{aligned}
\end{equation}
for an appropriate point $\ell^{\star}\in (c(\ell),\ell)$, where recall
\begin{equation}
c(\ell) \leqdef \ell - \frac{1.26 \, \tau^3 (d + 1) (\log n)^{3/2}}{G'(\ell) \, n^{1/2}}.
\end{equation}
It is easily verified that $t \mapsto G'(t)$ is decreasing for $t \geq 2 (d/2 - 1)$ (use Remark~\ref{rem:gamma.integral} together with the fact that $2 (d/2 - 1)$ is the mode of the Gamma distribution with shape parameter $d/2 - 1$ and scale parameter $2$), so that, by choosing $n = n(d,\tau)$ large enough and subsequently $\ell > c(\ell) \geq 2 (d/2 - 1)$, we have
\begin{equation}\label{eq:thm:multivariate.Tusnady.inequality.eq.middle}
G(c(\ell)) + \frac{1.26 \, \tau^3 (d + 1) (\log n)^{3/2}}{n^{1/2}} \leq G(\ell).
\end{equation}
From \eqref{eq:thm:multivariate.Tusnady.inequality.eq.middle} and \eqref{eq:thm:multivariate.Tusnady.inequality.eq.begin}, we deduce
\begin{equation}
G(c(\ell)) \leq F(\ell) = F\left(c(\ell) + \frac{1.26 \, \tau^3 (d + 1) (\log n)^{3/2}}{G'(\ell) \, n^{1/2}}\right).
\end{equation}
Therefore, by applying $F^{\star}$ on both sides, we get, for $n\geq n(d,\tau)$ large enough,
\begin{equation}
\Xi_{n,\bb{p}}(\omega) \leq \bb{\delta}_{\bb{Y}(\omega),\bb{p}}^{\top} \Sigma_{\bb{p}}^{-1} \bb{\delta}_{\bb{Y}(\omega),\bb{p}}^{\phantom{\top}} + \frac{1.26 \, \tau^3 (d + 1) (\log n)^{3/2}}{G'(L_{n,\bb{p},d,\tau}(\omega)) \, n^{1/2}}, \quad \omega\in A,
\end{equation}
where $L_{n,\bb{p},d,\tau}(\omega)$ solves $\bb{\delta}_{\bb{Y}(\omega),\bb{p}}^{\top} \Sigma_{\bb{p}}^{-1} \bb{\delta}_{\bb{Y}(\omega),\bb{p}}^{\phantom{\top}} = c(L_{n,\bb{p},d,\tau}(\omega))$ and $A$ is the event that we defined in \eqref{eq:event}. This ends the proof.
\end{proof}

\begin{proof}[Proof of Proposition~\ref{prop:optimization.problem.before}]
Throughout this proof, let $\ell^{\star} \leqdef \bb{\delta}_{\bb{n},\bb{p}}^{\top} \Sigma_{\bb{p}}^{-1} \bb{\delta}_{\bb{n},\bb{p}}^{\phantom{\top}}$. As in the proof of Theorem~\ref{thm.cdf.comparison}, the case $d = 1$ can easily be treated separately using the Berry-Esseen theorem (in fact $\e_n$ would be of the form $C n^{-1/2}$ in that case). Alternatively, we can prove our statement as we do below, for all $d\geq 1$ and $\ell^{\star} > 0$, only assuming further that $\ell^{\star} > 1/2$ when $d = 1$.

By the triangle inequality, we have
\begin{equation}\label{eq:diff.g.h.begin.2}
\begin{aligned}
&\Bigg|\sum_{\substack{\bb{k}\in \N_0^d \cap n \mathcal{S}_d \\ P_{n,\bb{p}}(\bb{k}) \geq P_{n,\bb{p}}(\bb{n})}} \hspace{-4mm} P_{n,\bb{p}}(\bb{k}) - \int_{\phi_{\Sigma_{\bb{p}}}(\bb{\delta}_{\bb{y},\bb{p}}) \geq \phi_{\Sigma_{\bb{p}}}(\bb{\delta}_{\bb{n},\bb{p}})} \frac{\phi_{\Sigma_{\bb{p}}}(\bb{\delta}_{\bb{y},\bb{p}})}{n^{d/2}} \rd \bb{y}\Bigg| \\[-1mm]
&\quad\leq \PP_{n,\bb{p}}(\mathscr{B}_{n,\bb{p}}^{\hspace{0.2mm}c}) + \left|\hspace{-1mm}
\begin{array}{l}
\PP_{n,\bb{p}}\left(\left\{\bb{k}\in \mathscr{B}_{n,\bb{p}} : P_{n,\bb{p}}(\bb{k}) \geq P_{n,\bb{p}}(\bb{n})\right\}\right) \\
- \widetilde{\PP}_{n,\bb{p}}\left(\left\{\bb{k}\in \mathscr{B}_{n,\bb{p}} : P_{n,\bb{p}}(\bb{k}) \geq P_{n,\bb{p}}(\bb{n})\right\} + \left[-\frac{1}{2},\frac{1}{2}\right]^d\right)
\end{array}
\hspace{-1mm}\right| \\
&\quad\quad+ \widetilde{\PP}_{n,\bb{p}}\left[\hspace{-1mm}
\begin{array}{l}
\left(\left\{\bb{k}\in \mathscr{B}_{n,\bb{p}} : P_{n,\bb{p}}(\bb{k}) \geq P_{n,\bb{p}}(\bb{n})\right\} + \left[-\frac{1}{2},\frac{1}{2}\right]^d\right) \\
\bigtriangleup \, \left(\left\{\bb{k}\in \mathscr{B}_{n,\bb{p}} : \bb{\delta}_{\bb{k},\bb{p}}^{\top} \Sigma_{\bb{p}}^{-1} \bb{\delta}_{\bb{k},\bb{p}}^{\phantom{\top}} \leq \ell^{\star}\right\} + \left[-\tfrac{1}{2},\tfrac{1}{2}\right]^d\right)
\end{array}
\hspace{-1mm}\right] \\
&\quad\quad+ |\widetilde{\PP}_{n,\bb{p}} - \QQ_{n,\bb{p}}|\left(\left\{\bb{k}\in \mathscr{B}_{n,\bb{p}} : \bb{\delta}_{\bb{k},\bb{p}}^{\top} \Sigma_{\bb{p}}^{-1} \bb{\delta}_{\bb{k},\bb{p}}^{\phantom{\top}} \leq \ell^{\star}\right\} + \left[-\tfrac{1}{2},\tfrac{1}{2}\right]^d\right) \\[1mm]
&\quad\quad+ \QQ_{n,\bb{p}}\left[\hspace{-1mm}
\begin{array}{l}
\left(\left\{\bb{k}\in \mathscr{B}_{n,\bb{p}} : \bb{\delta}_{\bb{k},\bb{p}}^{\top} \Sigma_{\bb{p}}^{-1} \bb{\delta}_{\bb{k},\bb{p}}^{\phantom{\top}} \leq \ell^{\star}\right\} + \left[-\tfrac{1}{2},\tfrac{1}{2}\right]^d\right) \\
\bigtriangleup \, \left\{\bb{y}\in \left(\mathscr{B}_{n,\bb{p}} + [-\tfrac{1}{2},\tfrac{1}{2}]^d\right) : \bb{\delta}_{\bb{y},\bb{p}}^{\top} \Sigma_{\bb{p}}^{-1} \bb{\delta}_{\bb{y},\bb{p}}^{\phantom{\top}} \leq \ell^{\star}\right\}
\end{array}
\hspace{-1mm}\right] \\
&\quad\quad+ \QQ_{n,\bb{p}}\left[\hspace{-1mm}
\begin{array}{l}
\left\{\bb{y}\in \left(\mathscr{B}_{n,\bb{p}} + [-\tfrac{1}{2},\tfrac{1}{2}]^d\right) : \bb{\delta}_{\bb{y},\bb{p}}^{\top} \Sigma_{\bb{p}}^{-1} \bb{\delta}_{\bb{y},\bb{p}}^{\phantom{\top}} \leq \ell^{\star}\right\} \\
\bigtriangleup \, \left\{\bb{y}\in \R^d : \bb{\delta}_{\bb{y},\bb{p}}^{\top} \Sigma_{\bb{p}}^{-1} \bb{\delta}_{\bb{y},\bb{p}}^{\phantom{\top}} \leq \ell^{\star}\right\}
\end{array}
\hspace{-1mm}\right] \\[1mm]
&\quad\reqdef (A)_{\ell^{\star}} + (B)_{\ell^{\star}} + (\bigstar)_{\ell^{\star}} + (C)_{\ell^{\star}} + (D)_{\ell^{\star}} + (E)_{\ell^{\star}}.
\end{aligned}
\end{equation}
The terms $(A)_{\ell^{\star}}$, $(B)_{\ell^{\star}}$, $(C)_{\ell^{\star}}$, $(D)_{\ell^{\star}}$ and $(E)_{\ell^{\star}}$ are already bounded in the proof of Theorem~\ref{thm.cdf.comparison}. It remains to bound $(\bigstar)_{\ell^{\star}}$.

If $\bb{k},\bb{n}\in \mathscr{B}_{n,\bb{p}}$, then an argument like \eqref{eq:bound.III.begin.1} shows that $P_{n,\bb{p}}(\bb{k}) \leq P_{n,\bb{p}}(\bb{n})$ and $\tfrac{1}{n^{d/2}} \phi_{\Sigma_{\bb{p}}}(\bb{\delta}_{\bb{k},\bb{p}}) > \tfrac{1}{n^{d/2}} \phi_{\Sigma_{\bb{p}}}(\bb{\delta}_{\bb{n},\bb{p}})$ imply
\begin{equation}
0 \leq \log\left\{\frac{n^{-d/2} \phi_{\Sigma_{\bb{p}}}(\bb{\delta}_{\bb{k},\bb{p}})}{n^{-d/2} \phi_{\Sigma_{\bb{p}}}(\bb{\delta}_{\bb{n},\bb{p}})}\right\} \leq 2 a_n, \quad \text{with } a_n \leqdef \frac{96.25 \, \tau^3 (\log n)^{3/2}}{n^{1/2}},
\end{equation}
and $P_{n,\bb{p}}(\bb{k}) < P_{n,\bb{p}}(\bb{n})$ and $\tfrac{1}{n^{d/2}} \phi_{\Sigma_{\bb{p}}}(\bb{\delta}_{\bb{k},\bb{p}}) \geq \tfrac{1}{n^{d/2}} \phi_{\Sigma_{\bb{p}}}(\bb{\delta}_{\bb{n},\bb{p}})$ imply
\begin{equation}
-2 a_n \leq \log\left\{\frac{n^{-d/2} \phi_{\Sigma_{\bb{p}}}(\bb{\delta}_{\bb{k},\bb{p}})}{n^{-d/2} \phi_{\Sigma_{\bb{p}}}(\bb{\delta}_{\bb{n},\bb{p}})}\right\} \leq 0.
\end{equation}
Therefore,
\begin{align}\label{eq:bound.bigstar.to.do}
(\bigstar)_{\ell^{\star}}
&\leq \widetilde{\PP}_{n,\bb{p}}\left(\left\{\bb{k}\in \mathscr{B}_{n,\bb{p}} :\left|\log\left\{\frac{n^{-d/2} \phi_{\Sigma_{\bb{p}}}(\bb{\delta}_{\bb{k},\bb{p}})}{n^{-d/2} \phi_{\Sigma_{\bb{p}}}(\bb{\delta}_{\bb{n},\bb{p}})}\right\}\right| \leq 2 a_n\right\} + \left[-\tfrac{1}{2},\tfrac{1}{2}\right]^d\right) \notag \\
&\leq \widetilde{\PP}_{n,\bb{p}}\left(\left\{\bb{k}\in \mathscr{B}_{n,\bb{p}} : \big|\bb{\delta}_{\bb{k},\bb{p}}^{\top} \Sigma_{\bb{p}}^{-1} \bb{\delta}_{\bb{k},\bb{p}}^{\phantom \top} - \ell^{\star}\big| \leq 4 a_n\right\} + \left[-\tfrac{1}{2},\tfrac{1}{2}\right]^d\right).
\end{align}
By the triangle inequality, we can then split the above probability into three distinct terms:
\begin{align}\label{eq:bound.bigstar.to.do.split.three}
&\leq \big|\widetilde{\PP}_{n,\bb{p}} - \QQ_{n,\bb{p}}\big|\left(\left\{\bb{k}\in \mathscr{B}_{n,\bb{p}} : \big|\bb{\delta}_{\bb{k},\bb{p}}^{\top} \Sigma_{\bb{p}}^{-1} \bb{\delta}_{\bb{k},\bb{p}}^{\phantom \top} - \ell^{\star}\big| \leq 4 a_n\right\} + \left[-\tfrac{1}{2},\tfrac{1}{2}\right]^d\right) \notag \\
&\quad+ \QQ_{n,\bb{p}}\left[\hspace{-1mm}
\begin{array}{l}
\left(\left\{\bb{k}\in \mathscr{B}_{n,\bb{p}} : \big|\bb{\delta}_{\bb{k},\bb{p}}^{\top} \Sigma_{\bb{p}}^{-1} \bb{\delta}_{\bb{k},\bb{p}}^{\phantom \top} - \ell^{\star}\big| \leq 4 a_n\right\} + \left[-\tfrac{1}{2},\tfrac{1}{2}\right]^d\right) \\[1.5mm]
\bigtriangleup \, \left\{\bb{y}\in \left(\mathscr{B}_{n,\bb{p}} + [-\tfrac{1}{2},\tfrac{1}{2}]^d\right) : \big|\bb{\delta}_{\bb{y},\bb{p}}^{\top} \Sigma_{\bb{p}}^{-1} \bb{\delta}_{\bb{y},\bb{p}} - \ell^{\star}\big| \leq 4 a_n\right\}
\end{array}
\hspace{-1mm}\right] \notag \\[0.5mm]
&\quad+ \QQ_{n,\bb{p}}\left\{\bb{y}\in \left(\mathscr{B}_{n,\bb{p}} + [-\tfrac{1}{2},\tfrac{1}{2}]^d\right) : \big|\bb{\delta}_{\bb{y},\bb{p}}^{\top} \Sigma_{\bb{p}}^{-1} \bb{\delta}_{\bb{y},\bb{p}} - \ell^{\star}\big| \leq 4 a_n\right\} \notag \\[1mm]
&\reqdef (\bigstar_1)_{\ell^{\star}} + (\bigstar_2)_{\ell^{\star}} + (\bigstar_3)_{\ell^{\star}}.
\end{align}
For the first term in \eqref{eq:bound.bigstar.to.do.split.three}, we can use the total variation bound in Proposition~\ref{prop:total.variation.bound}:
\begin{equation}\label{eq:bound.bigstar.1}
(\bigstar_1)_{\ell^{\star}} \leq \frac{8.03 \, \tau^{3/2} (d + 1)^{1/2}}{n^{1/2}}.
\end{equation}
Bounding the second term in \eqref{eq:bound.bigstar.to.do.split.three} is the same as $(D)_{\ell^{\star}}$ except that there are two ellipses instead of one, namely $\{\bb{y}\in \R^d : \bb{\delta}_{\bb{y},\bb{p}}^{\top} \Sigma_{\bb{p}}^{-1} \bb{\delta}_{\bb{y},\bb{p}} = \ell^{\star} + 4 a_n\}$ and $\{\bb{y}\in \R^d : \bb{\delta}_{\bb{y},\bb{p}}^{\top} \Sigma_{\bb{p}}^{-1} \bb{\delta}_{\bb{y},\bb{p}} = \ell^{\star} - 4 a_n\}$. Therefore,
\begin{equation}\label{eq:bound.bigstar.2}
(\bigstar_2)_{\ell^{\star}} \leq 2 \cdot b_n = \frac{2.10 \, \tau^2 (d + 1)^2 (\log n)^{1/2}}{n^{1/2}}.
\end{equation}
For the third term in \eqref{eq:bound.bigstar.to.do.split.three}, we can use Equation~\eqref{eq:bound.diff.ellipses}, Remark~\ref{rem:gamma.integral} and the mean value theorem (only when $d\geq 2$) to obtain
\begin{align}\label{eq:bound.bigstar.3.begin}
(\bigstar_3)_{\ell^{\star}}
&\leq \PP(\bb{\delta}_{\bb{Y}\hspace{-0.5mm},\bb{p}}^{\top} \Sigma_{\bb{p}}^{-1} \bb{\delta}_{\bb{Y}\hspace{-0.5mm},\bb{p}}^{\phantom{\top}} \leq \ell^{\star} + 4 a_n) - \PP\{\bb{\delta}_{\bb{Y}\hspace{-0.5mm},\bb{p}}^{\top} \Sigma_{\bb{p}}^{-1} \bb{\delta}_{\bb{Y}\hspace{-0.5mm},\bb{p}}^{\phantom{\top}} \leq \max(0,\ell^{\star} - 4 a_n)\} \notag \\[2mm]
&= \overline{\gamma}\{d/2,(\ell^{\star}+4 a_n)/2\} - \overline{\gamma}\{d/2,\max(0,\ell^{\star}-4 a_n)/2\} \\
&\leq \max_{t\in [0,\infty) : |2t - \ell^{\star}| \leq 4 a_n} \frac{t^{d/2-1} e^{-t}}{\Gamma(d/2)} \cdot 4 a_n
\end{align}
Now, notice that the maximum above is attained at $t = \max(0,\ell^{\star}/2 - b_n/2)$ for $d\in \{1,2\}$. More generally, the function $t\mapsto t^{d/2-1} e^{-t}$ maximizes on $(0,\infty)$ at $t = d/2 - 1$ for all $d\geq 3$, and
\begin{equation}
\Gamma(d/2) = (d/2 - 1) \Gamma(d/2 - 1) \geq \sqrt{2\pi} (d/2 - 1)^{d/2 + 1/2} e^{-(d/2 - 1)}, \quad \text{for all } d \geq 3,
\end{equation}
by a standard Stirling lower bound (see, e.g., Lemma~1 of \cite{MR162751}, or Theorem~2.2 of \cite{MR3684463} for a more precise bound). Using this in \eqref{eq:bound.E.begin} shows that, for any given $\ell^{\star} > 0$ (we assumed that $\ell^{\star} > 1/2$ when $d = 1$),
\begin{align}\label{eq:bound.bigstar.3}
(\bigstar_3)_{\ell^{\star}}
\leq \left\{\hspace{-1mm}
\begin{array}{ll}
\frac{\{\max(0,\ell^{\star}/2 - b_n/2)\}^{-1/2}}{\sqrt{\pi} \, e^{\max(0,\ell^{\star}/2 - b_n/2)}} \cdot b_n, &\mbox{if } d = 1 \\
e^{-\max(0,\ell^{\star}/2 - b_n/2)} \cdot b_n, &\mbox{if } d = 2 \\[1mm]
\frac{1}{\sqrt{2\pi} (d/2  - 1)} \cdot b_n, &\mbox{if } d\geq 3
\end{array}
\hspace{-1mm}\right\} \leq 4 a_n = \frac{385 \, \tau^3 (\log n)^{3/2}}{n^{1/2}}.
\end{align}

By putting \eqref{eq:bound.bigstar.1}, \eqref{eq:bound.bigstar.2} and \eqref{eq:bound.bigstar.3} together in \eqref{eq:bound.bigstar.to.do}, we get
\begin{align}\label{eq:bound.bigstar}
(\bigstar)_{\ell^{\star}}
&\leq (\bigstar_1)_{\ell^{\star}} + (\bigstar_2)_{\ell^{\star}} + (\bigstar_3)_{\ell^{\star}} \notag \\[1mm]
&\leq \frac{8.03 \, \tau^{3/2} (d + 1)^{1/2}}{n^{1/2}} + \frac{2.10 \, \tau^2 (d + 1)^2 (\log n)^{1/2}}{n^{1/2}} + \frac{385 \, \tau^3 (\log n)^{3/2}}{n^{1/2}} \notag \\
&\leq \frac{193.70 \, \tau^3 (d + 1) (\log n)^{3/2}}{n^{1/2}}.
\end{align}

Finally, by putting the bounds for $(A)_{\ell^{\star}}$, $(B)_{\ell^{\star}}$, $(\bigstar)_{\ell^{\star}}$, $(C)_{\ell^{\star}}$, $(D)_{\ell^{\star}}$, $(E)_{\ell^{\star}}$ found in \eqref{eq:bound.A}, \eqref{eq:bound.B}, \eqref{eq:bound.bigstar}, \eqref{eq:bound.C}, \eqref{eq:bound.D} and \eqref{eq:bound.E}, respectively, together in \eqref{eq:diff.g.h.begin.2}, we get the bound
\begin{align}
&\left|\sum_{\substack{\bb{k}\in \N_0^d \cap n \mathcal{S}_d \\ P_{n,\bb{p}}(\bb{k}) \geq P_{n,\bb{p}}(\bb{n})}} \hspace{-4mm} P_{n,\bb{p}}(\bb{k}) - \int_{\phi_{\Sigma_{\bb{p}}}(\bb{\delta}_{\bb{y},\bb{p}}) \geq \phi_{\Sigma_{\bb{p}}}(\bb{\delta}_{\bb{n},\bb{p}})} \frac{\phi_{\Sigma_{\bb{p}}}(\bb{\delta}_{\bb{y},\bb{p}})}{n^{d/2}} \rd \bb{y}\right| \\[1mm]
&\quad\leq (A)_{\ell^{\star}} + (B)_{\ell^{\star}} + (\bigstar)_{\ell^{\star}} + (C)_{\ell^{\star}} + (D)_{\ell^{\star}} + (E)_{\ell^{\star}} \notag \\[0.5mm]
&\quad\leq \frac{2 \, (d + 1)}{n^2} + 0 + \frac{193.70 \, \tau^3 (d + 1) (\log n)^{3/2}}{n^{1/2}} + \frac{8.03 \, \tau^{3/2} (d + 1)^{1/2}}{n^{1/2}} \notag \\[0.5mm]
&\qquad+ \frac{1.05 \, \tau^2 (d + 1)^2 (\log n)^{1/2}}{n^{1/2}} + \frac{8.07 \, \tau^{3/2} (d + 1)^{1/2}}{n^{1/2}} \notag \\[0.5mm]
&\quad\leq \frac{194.96 \, \tau^3 (d + 1) (\log n)^{3/2}}{n^{1/2}}.
\end{align}
This ends the proof.
\end{proof}

\begin{proof}[Proof of Theorem~\ref{thm:solution}]
Equation~\eqref{eq:problem.description.general.next} is direct consequence of Proposition~\ref{prop:optimization.problem.before}. It remains to show that the set $\widetilde{\mathcal{C}}(\hat{\bb{p}}_n, \alpha - \e_n)$ is strictly convex.

Using the calculation in \eqref{eq:Pearson.calculation}, we can write
\begin{align}
\frac{1}{2} \bb{\delta}_{\bb{n},\bb{p}}^{\top} \Sigma_{\bb{p}}^{-1} \bb{\delta}_{\bb{n},\bb{p}}^{\phantom \top}
&= \frac{n}{2} \sum_{i=1}^d \frac{(\hat{p}_i - p_i)^2}{p_i} + \frac{n}{2} \frac{(\hat{p}_{d+1} - p_{d+1})^2}{p_{d+1}} \notag \\
&= \frac{n}{2} \sum_{i=1}^d \left(\frac{\hat{p}_i^2}{p_i} - 2 \hat{p}_i + p_i\right) + \frac{n}{2} \left(\frac{\hat{p}_{d+1}^2}{p_{d+1}} - 2 \hat{p}_{d+1} + p_{d+1}\right),
\end{align}
so that the Hessian matrix of $\bb{p}\mapsto \tfrac{1}{2} \bb{\delta}_{\bb{n},\bb{p}}^{\top} \Sigma_{\bb{p}}^{-1} \bb{\delta}_{\bb{n},\bb{p}}^{\phantom \top}$ is positive definite on $(0,\infty)^d$. Indeed,
\begin{equation}
\frac{\rd^2}{\rd \bb{p} \, \rd \bb{p}^{\top}} \, \frac{1}{2} \bb{\delta}_{\bb{n},\bb{p}}^{\top} \Sigma_{\bb{p}}^{-1} \bb{\delta}_{\bb{n},\bb{p}}^{\phantom \top} = n \left[\mathrm{diag}\left\{\left(\frac{\hat{p}_i^2}{p_i^3}\right)_{\hspace{-0.5mm}i=1}^{\hspace{-0.5mm}d}\right\} + \frac{\hat{p}_{d+1}^2}{p_{d+1}^3} \bb{1}_d^{\phantom{\top}}\bb{1}_d^{\top}\right] > 0.
\end{equation}
Therefore, $\bb{p}\mapsto \tfrac{1}{2} \bb{\delta}_{\bb{n},\bb{p}}^{\top} \Sigma_{\bb{p}}^{-1} \bb{\delta}_{\bb{n},\bb{p}}^{\phantom \top}$ is strictly convex, and sublevel sets of the form
\begin{equation}
\left\{\bb{p}\in (0,\infty)^d : \frac{1}{2} \bb{\delta}_{\bb{n},\bb{p}}^{\top} \Sigma_{\bb{p}}^{-1} \bb{\delta}_{\bb{n},\bb{p}}^{\phantom \top} \leq L\right\}
\end{equation}
are strictly convex for any $L > 0$. Since $y\mapsto \gamma(d/2,y)$ is invertible ($y\mapsto \gamma(d/2,y)$ is clearly increasing), write the inverse as $\phi_d(\cdot)$. The domain over which we optimize the objective function in \eqref{eq:compute.mu.star.d.general} is the intersection of two strictly convex sets, namely
\begin{equation}\label{eq:last}
\widetilde{\mathcal{C}}(\hat{\bb{p}}_n, \alpha + \e_n) = \mathcal{S}_d^o \cap \left\{\bb{p}\in (0,\infty)^d : \frac{1}{2} \bb{\delta}_{\bb{n},\bb{p}}^{\top} \Sigma_{\bb{p}}^{-1} \bb{\delta}_{\bb{n},\bb{p}}^{\phantom \top} \leq \phi_d(1 - \alpha + \e_n)\right\},
\end{equation}
which proves that $\widetilde{\mathcal{C}}(\hat{\bb{p}}_n, \alpha + \e_n)$ is strictly convex. As pointed in Remark~\ref{rem:unique}, the fact that the first term in the intersection is $\mathcal{S}_d^o$ plays a crucial role for $\widetilde{\mathcal{C}}(\hat{\bb{p}}_n, \alpha + \e_n)$ to be strictly convex and thus for the uniqueness minimizer/maximizer of the objective function.
\end{proof}

\section{Moments estimates}\label{sec:moments.multinomial}

In this section, we derive some moment estimates for the multinomial distribution. The lemma below is used in the proof of Proposition~\ref{prop:total.variation.bound}.

\begin{lemma}\label{lem:moments.multinomial}
Let $n\in \N$ and $\bb{p}\in \mathcal{S}_d^o$ be given.
If $\bb{\xi} = (\xi_1,\xi_2,\dots,\xi_d)\sim \mathrm{Multinomial}\hspace{0.2mm}(n,\bb{p})$ according to \eqref{eq:multinomial.pdf}, then, for all $1 \leq i \leq d + 1$,
\begin{align}
&\EE\left(|\delta_{\xi_i,p_i}|^2\right) = p_i (1 - p_i), \label{eq:thm:central.moments.2.0} \\[1mm]
&\EE\left(|\delta_{\xi_i,p_i}|^4\right) = 3 p_i^2 (1 - p_i)^2 + n^{-1} p_i \, (1 - 7 p_i + 12 p_i^2 - 6 p_i^3), \label{eq:thm:central.moments.4.0}
\end{align}
where $\xi_{d+1} \leqdef n - \|\bb{\xi}\|_1$. Moreover, for any Borel set $B\in \mathscr{B}(\R^d)$ and all $n \geq 16$, we have
\begin{align}
&\Big|\EE\left(\delta_{\xi_i,p_i} \, \ind_{\{\bb{\xi}\in B\}}\right)\Big| \leq \frac{1}{2} \left\{\PP(\bb{\xi}\in B^c)\right\}^{1/2}, \label{eq:thm:central.moments.eq.1.set.B} \\
&\left|\EE\left(\delta_{\xi_i,p_i}^3 \, \ind_{\{\bb{\xi}\in B\}}\right) - n^{-1/2} p_i \, (2 p_i^2 - 3 p_i + 1)\right| \leq \frac{1}{\sqrt{8}} \left\{\PP(\bb{\xi}\in B^c)\right\}^{1/4}, \label{eq:thm:central.moments.eq.3.set.B}
\end{align}
\end{lemma}

\begin{proof}
The result \eqref{eq:thm:central.moments.2.0} is well-known and the result \eqref{eq:thm:central.moments.4.0} is shown for example in Equation~(A.29) of \cite{MR4249129} or Equation~(82) in \cite{Ouimet_2021_multinomial_moments}. The result \eqref{eq:thm:central.moments.eq.1.set.B} is a direct consequence of \eqref{eq:thm:central.moments.2.0} together with the Cauchy-Schwarz inequality and the fact that $\max_{p\in [0,1]} p (1 - p) = 1/4$. Finally, to obtain \eqref{eq:thm:central.moments.eq.3.set.B}, we use the fact that $\EE(\delta_{\xi_i,p_i}^3) = n^{-1/2} p_i \, (2 p_i^2 - 3 p_i + 1)$ (see Lemma~A.1 in \cite{MR4249129} or Equation~(79) in \cite{Ouimet_2021_multinomial_moments}) followed by an application of H\"older's inequality and the bound $\EE(|\delta_{\xi_i,p_i}|^4) \leq 1/4$ (which follows from \eqref{eq:thm:central.moments.4.0} and the assumption $n\geq 16$):
\begin{align}\label{eq:thm:central.moments.eq.3.set.B.proof}
&\left|\EE\left(\delta_{\xi_i,p_i}^3 \, \ind_{\{\bb{\xi}\in B\}}\right) - n^{-1/2} p_i \, (2 p_i^2 - 3 p_i + 1)\right| \notag \\[0.5mm]
&\quad= \left|\EE\left(\delta_{\xi_i,p_i}^3 \, \ind_{\{\bb{\xi}\in B^c\}}\right)\right| \notag \\[0.5mm]
&\quad\leq \left\{\EE\left(|\delta_{\xi_i,p_i}|^4\right)\right\}^{1/4} \left\{\EE\left(|\delta_{\xi_i,p_i}|^4\right)\right\}^{1/4} \left\{\EE\left(|\delta_{\xi_i,p_i}|^4\right)\right\}^{1/4} \left\{\PP(\bb{\xi}\in B^c)\right\}^{1/4} \notag \\[1.5mm]
&\quad\leq \left(1/4\right)^{1/4} \left(1/4\right)^{1/4} \left(1/4\right)^{1/4} \left\{\PP(\bb{\xi}\in B^c)\right\}^{1/4} \notag \\
&\quad= \frac{1}{\sqrt{8}} \left\{\PP(\bb{\xi}\in B^c)\right\}^{1/4}.
\end{align}
This ends the proof.
\end{proof}

\end{appendices}

\section*{R code}

The \texttt{R} code that generated Figures~\ref{fig:example.1}~and~\ref{fig:example.2} is available at \\ \href{https://www.dropbox.com/s/tfwnnba642740xo/multinomial_method.R?dl=0}{https://www.dropbox.com/s/tfwnnba642740xo/multinomial\_method.R?dl=0}.



\section*{Funding}

F.\ Ouimet was supported by a CRM-Simons postdoctoral fellowship from the Centre de recherches math\'ematiques (Montr\'eal, Canada) and the Simons Foundation.



%
%

\phantomsection
\addcontentsline{toc}{chapter}{References}

\bibliographystyle{authordate1}
\bibliography{Bax_Ouimet_2023_Pearson_approx_bib}

\end{document}